\documentclass[12pt,twoside,leqno]{article}

\usepackage{latexsym, amsfonts, amsmath, amssymb,amsthm,verbatim,amscd,tikz,enumitem}
\usepackage{authblk}
\usepackage{scalerel}

\newtheorem{theorem}{Theorem}[section]
\newtheorem{corollary}[theorem]{Corollary}
\newtheorem{lemma}[theorem]{Lemma}

\theoremstyle{definition}
\newtheorem{definition}[theorem]{Definition}
\newtheorem{example}[theorem]{Example}

\newtheorem{question}[theorem]{Question}

\newcommand{\closeup}{\parskip 0pt}

\newcommand{\AC}{{AC}}
\newcommand{\BV}{{BV}}

\newcommand{\CTPP}{\mathrm{CTPP}} 
\newcommand{\PIC}{\mathrm{PIC}}
\newcommand{\APIC}{\mathrm{APIC}}
\newcommand{\SPIC}{\mathrm{SPIC}}

\newcommand{\PWP}{PW^+}

\newcommand{\calC}{\mathcal{C}}
\newcommand{\calM}{\mathcal{M}}

\newcommand{\mR}{\mathbb{R}}
\newcommand{\mC}{\mathbb{C}}

\newcommand{\mT}{\mathbb{T}}

\newcommand{\vecx}{{\boldsymbol{x}}}
\newcommand{\vecy}{{\boldsymbol{y}}}
\newcommand{\vecu}{{\boldsymbol{u}}}

\newcommand{\vecw}{{\boldsymbol{w}}}
\newcommand{\vecz}{{\boldsymbol{z}}}

\def\ls[#1,#2]{\overline{#1\,#2}}

\newcommand{\hatf}{{\hat f}}
\newcommand{\norm}[1]{\left\lVert#1\right\rVert}

\newcommand{\ssnorm}[1]{\lVert #1 \rVert}

\newcommand{\normbv}[1]{\left\lVert#1\right\rVert_{\BV(\sigma)}}

\def\ipr<#1,#2>{\langle #1,#2\rangle}
\def\ls[#1,#2]{\overline{\vphantom{\vbox to 1.2 ex{}} #1\, #2}}

\newcommand{\cl}{\mathop{\mathrm{cl}}}

\newcommand{\st}{\,:\,}

\DeclareMathOperator*{\var}{var}
\DeclareMathOperator*{\cvar}{\rm cvar}

\DeclareMathOperator{\vf}{vf}

\newcommand{\intr}{{\rm int}}

\newcommand{\journalname}[1]{\textrm{#1}}

\begin{document}

\baselineskip=17pt

\title{Tietze type extensions for absolutely continuous functions in the plane}

\author{Ian Doust\\
\and
Alan Stoneham\\
School of Mathematics and Statistics\\
University of New South Wales\\
UNSW Sydney 2052\\
Australia\\ \ \\
Email: I.Doust@unsw.edu.au  \\
Email: A.Stoneham@unsw.edu.au}

\date{}

\maketitle

\renewcommand{\thefootnote}{}



\renewcommand{\thefootnote}{\arabic{footnote}}
\setcounter{footnote}{0}

%
%

\begin{abstract}
It is an open problem whether one can always extend an absolutely continuous function (in the sense of Ashton and Doust) on a compact subset of the plane to a larger compact set. In this paper we show that this can be done for a large family of initial domains whose components consist of polygons and convex curves. An application is given to the spectral theory of $AC(\sigma)$ operators.
\end{abstract}

\vfill\eject
%
%
\section{Introduction}

A Hilbert space operator $T$ is normal if and only if it admits a $C(\sigma(T))$ functional calculus. Indeed, if there is any nonempty compact set $\sigma \subseteq \mC$ and a $C^*$-algebra isomorphism $\Phi: C(\sigma) \to B(H)$ such that $\Psi$ maps the identity function on $\sigma$ to $T$, then $T$ is normal. In this case $\sigma(T) \subseteq \sigma$ and the 
kernel of $\Phi$ is just the set of continuous functions which vanish on $\sigma(T)$.

In \cite{AD3} Doust and Ashton introduced a class of Banach space operators called $\AC(\sigma)$ operators which generalize the classes of well-bounded operators and trigonometrically well-bounded operators. In the setting of Hilbert spaces, all self-adjoint operators are well-bounded and all unitary operators are trigonometrically well-bounded, and in this context $AC(\sigma)$ operators may be seen as an analogue of normal operators (see, for example, \cite{Dow} and \cite{BG1,BG2,BG3} for more details of these classes of operators). An operator is an $AC(\sigma)$ operator if it admits a functional calculus for an algebra $AC(\sigma)$ of `absolutely continuous' functions on $\sigma$ which was introduced in \cite{AD1}. (Details of the spaces $AC(\sigma)$ for compact sets $\sigma \subseteq \mC$ will be given in Section~\ref{AC-properties}.) 

If $T$ is an $AC(\sigma)$ operator, then certainly $\sigma(T) \subseteq \sigma$, but at present it remains an open problem as to whether $T$ must in fact be an $AC(\sigma(T))$ operator. 
Such questions are closely related to classical extension theorems. Suppose that $\sigma_0 \subseteq \sigma$ are two compact subsets of the plane. The Tietze Extension Theorem ensures that if $f \in C(\sigma_0)$ then there exists $\hatf \in C(\sigma)$ such that $f = \hatf|\sigma_0$ and such that $\norm{f}_\infty = \ssnorm{\hatf}_\infty$. The main focus of this paper concerns the open question as to whether a corresponding extension result holds for absolutely continuous functions.
In both settings, the restriction map $g \mapsto g|\sigma_0$ is a norm 1 algebra homomorphism. The extension question is asking whether this map is onto. 

The above fact about the restriction map means that if one can extend a function on $\sigma_0$ to any suitably large square or disk, then one can extend it to any compact superset. This means that the primary issue in such questions is the nature of the original domain $\sigma_0$, rather than the nature of $\sigma$. 

For certain classes of sets, an extension of an absolutely continuous function is always possible. This is easily seen to be the case if, for example, $\sigma_0$ is a compact subset of the real line.  On the other hand, as the following example shows, not every `natural' extension of an absolutely continuous function is absolutely continuous. 

\begin{example}\label{Cantor-ex}
Let $C: [0,1] \to [0,1]$ denote the Cantor function. Let $\sigma_0 = \{(x,C(x)) \st x \in [0,1]\}$, let $\sigma = [0,1] \times [0,1]$ and let $f: \sigma_0 \to \mC$ be defined by $f(x,y) = y$, $(x,y) \in \sigma_0$. Then (by the definitions in the next section) $f \in AC(\sigma_0)$. One might consider both the  functions $\hatf_1(x,y) = y$ and $\hatf_2(x,y) = C(x)$ as natural continuous extensions of $f$ to $\sigma$, but only the first of these is absolutely continuous.
\end{example}

A few particular extension results for absolutely continuous functions have appeared in the literature. For example, a formula which extends an absolutely continuous function on the boundary of a square to its interior was given in \cite[Theorem~3.5]{AD3}. 


\begin{figure}
\begin{center}
\begin{tikzpicture}[scale=0.8]

\draw[fill,black] (1,5) -- (5,5) -- (5,7) -- (3,7) -- (3,9) -- (1,9) -- (1,5);
\draw[fill,white] (1.6,7.4) -- (2.4,7.4) -- (2.4,8.2) -- (1.6,8.2) -- (1.6,7.4);
\draw[fill,white] (1.6,5.5) -- (4.4,5.5) -- (4.4,6.5) -- (1.6,6.5) -- (1.6,5.5);
\draw[fill,black] (2.5,5.8) -- (3.5,5.8) -- (3.5,6.2) -- (2.5,6.2) -- (2.5,5.8);

\draw[fill,black] (12,6) -- (13,6) -- (13,7) -- (12,7) -- (12,6);

\draw[black,ultra thick] (13,6.5) circle (2cm);
\draw[black,ultra thick] (13.5,6.5) circle (0.7cm);
\draw[black,ultra thick] (9,8) arc (90:180:2cm);
\draw[black,ultra thick] (7,5) -- (7,8);
\draw[black,fill]  (10., 7.7321) circle (0.05cm);
\draw[black,fill]  (10.6180,  7.1756) circle (0.05cm);
\draw[black,fill]  (10.8019, 6.8678 ) circle (0.05cm);
\draw[black,fill]    (10.8794, 6.6840) circle (0.05cm);
\draw[black,fill]    (10.9190, 6.5635 ) circle (0.05cm);
\draw[black,fill]    (10.9419,  6.4786) circle (0.05cm);
\draw[black,fill]    (10.9563,  6.4158) circle (0.05cm);
\draw[black,fill]    (10.9659,  6.3675) circle (0.05cm);
\draw[black,fill]    (10.9727, 6.3292) circle (0.05cm);
\draw[black,fill]    (10.9777,  6.2981) circle (0.05cm);

\draw[black,ultra thick] (1.8,5.7) -- (2.2,6.1);
\draw[black,ultra thick] (1.8,6.1) -- (2.2,5.7);
\draw[black,ultra thick] (4.4,6) circle (0.5cm);
\draw[black,ultra thick] (3.9,6) -- (7,6);

\end{tikzpicture}
\caption{An example of a plane set $\sigma_0$ which will satisfy the hypotheses of the main extension result, Theorem~\ref{main-theorem}.}\label{Example1}
\end{center}
\end{figure}
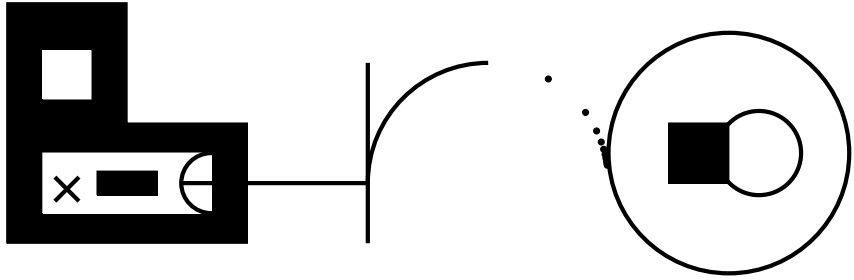

The main result of this paper, Theorem~\ref{main-theorem}, gives sufficient conditions on $\sigma_0$ which ensure that absolutely continuous extensions to supersets always exist. 
The main theorem covers sets such as the one in Figure~\ref{Example1} which are made up of polygonal regions with polygonal holes and convex curves satisfying certain hypotheses. Since the sets may be made up of pieces of a somewhat different nature, the proof essentially progresses by dealing with different parts of the set separately and then patching the pieces together. This will involve proving a number of `joining theorems'. 

These theorems, which are of independent interest, examine when one can say that a function which is absolutely continuous on each of two sets is then absolutely continuous on their union. We give some examples to show that even when the two sets are quite simple, such a conclusion may not hold.

In the simplest situation, where $\sigma_0 \subseteq \sigma \subseteq \mR$, one can always extend an absolutely continuous function on $\sigma_0$ to one on $\sigma$ without an increase in the variation norm. For the plane sets considered in this paper, it is not known whether one can always find an extension with the same norm. Certainly our methods do not produce such an extension. On the other hand we are able to at least show that the norm of the extension is controlled by some constant (depending on $\sigma_0$) times the norm of the function on $\sigma_0$. In order to keep the proof as simple as possible we have generally made no attempt to optimize the bounds, although in a few places we shall note that where our methods require that the constants that appear are greater than one.

There are a number of relatively simple sets for which it is not known whether extensions always exist. These include closed disks and the graph of the Cantor function considered above. 

In the next section we shall recall the main definitions and properties of absolutely continuous functions. The following section will introduce the classes of compact sets which appear in the main result, Theorem~\ref{main-theorem}. 

Although the great majority of the paper is purely concerned with the algebras of absolutely continuous functions, we shall return in the final section to revisit the original motivating question concerning $AC(\sigma)$ operators referred to above. Due to this, we shall work throughout with complex-valued functions. However, all the results remain true for real-valued functions. 


\section{Definitions and properties of $AC(\sigma)$}\label{AC-properties}

Throughout this paper, $\sigma_0$ and $\sigma$ will denote nonempty compact subsets of the plane. We shall consider the plane to be either $\mR^2$ or $\mC$ as is notationally convenient, and identify the real line with the $x$-axis in the plane.

For the moment then, fix $\sigma \subseteq \mR^2$. We shall now briefly recall the definition of the Banach algebra of functions of bounded variation on $\sigma$ and of its subalgebra, the set of absolutely continuous functions on $\sigma$.  Further details and proofs may be found in \cite{DLS1}.

Suppose that $f: \sigma \to \mC$. Given an ordered finite list of (not necessarily distinct) elements $S= [\vecx_0,\dots,\vecx_n]$ in $\sigma$, we set
  \[ \cvar(f,S) = \sum_{k=1}^n |f(\vecx_k) - f(\vecx_{k-1})|. \]
The variation factor of the list $S$, denoted $\vf(S)$, is a positive integer which roughly measures the maximum number of times that the piecewise linear path joining the points in $S$ in order crosses any line. More formally, let $\gamma_S: [0,1] \to \mC$ be a parameterization of the piecewise linear curve joining the elements in $S$ in order. Given a line $\ell$ in the plane, let $L(S,\ell) = \gamma_S^{-1}(\ell) \subseteq [0,1]$ and let $\vf(S,\ell)$ denote the number of connected components of $L(S,\ell)$. The \textbf{variation factor of $S$} is then $\vf(S) = \sup_\ell \vf(S,\ell)$. The main fact that we shall need in the later sections is that if $P$ is an simple $m$-gon and $S = [\vecx_0,\dots,\vecx_n]$ has $k$ line segments $\ls[\vecx_{j-1},\vecx_j]$ with one endpoint inside $P$ and one endpoint outside $P$, then $\vf(S) \ge \lceil k/m \rceil$. We refer the reader to \cite[Section 2.1]{DLS1} and the appendix to \cite{AS3} for further details of the properties of the variation factor.

The \textbf{variation of $f$ over $\sigma$} is found by taking a supremum over all such finite lists of points:
  \[ \var(f,\sigma) = \sup_{S} \frac{\cvar(f,S)}{\vf(S)}. \]
The  space of functions of bounded variation on $\sigma$, denoted $BV(\sigma)$, consists of all functions $f: \sigma \to \mC$ for which
  \[ \norm{f}_{BV(\sigma)}  = \sup_{\vecx \in \sigma} |f(\vecx)| + \var(f,\sigma) \]
is finite. 

Calculating $\var(f,\sigma)$ precisely is often challenging, so we shall mainly be using the properties of variation in order to obtain good bounds. We recall here the most important of these properties. Further details can be found in \cite{DLS1}.

\begin{theorem}\label{bv-properties}
Suppose that $\sigma$ is a nonempty compact subset of $\mR^2$. Then
\begin{enumerate}[label=(\roman*)]
\item $BV(\sigma)$ is a Banach algebra under the norm $\norm{\cdot}_{BV(\sigma)}$.
\item $BV(\sigma)$ contains the (restrictions of) complex polynomials in two real variables. 
\item If $\sigma = [a,b] \subseteq \mR$ then $BV(\sigma)$ is the usual classical space of functions of bounded variation, and $\var(f,[a,b])$ is just the usual measure of the variation of a function on the interval $[a,b]$.
\item If $\phi$ is an invertible affine transformation of the plane and $\sigma' = \phi(\sigma)$, then $\Phi(f) = f \circ \phi^{-1}$ is an isometric isomorphism from $BV(\sigma)$ to $BV(\sigma')$.
\end{enumerate}
\end{theorem}


\begin{definition}
The space $AC(\sigma)$ of \textbf{absolutely continuous functions} on $\sigma$ is the closure of the complex polynomials in two real variables in $BV(\sigma)$.
\end{definition}

%

We record some of the most important properties of $AC(\sigma)$ spaces.

\begin{theorem}\label{ac-properties}
Suppose that $\sigma$ is a nonempty compact subset of $\mR^2$.
\begin{enumerate}[label=(\roman*)]
\item If $\sigma = [a,b] \subseteq \mR$ then $AC(\sigma)$ is the usual classical space of absolutely continuous functions.
  \item\label{aff-inv} (Affine invariance) If $\phi$ is an invertible affine transformation of the plane and $\sigma' = \phi(\sigma)$, then $\Phi(f) = f \circ \phi^{-1}$ is an isometric isomorphism from $AC(\sigma)$ to $AC(\sigma')$.
\item\label{restr} If $\sigma_0$ is a nonempty compact subset of $\sigma$ and $f  \in AC(\sigma)$, then $f|\sigma_0 \in AC(\sigma_0)$ with $\norm{f|\sigma_0}_{BV(\sigma_0)} \le \norm{f}_{BV(\sigma)}$.
\item\label{r-ext} If $\sigma_x$ denotes the projection of $\sigma$ onto the $x$-axis, and $f\in AC(\sigma_x)$, then $\hat{f}:\sigma\to\mC$ defined by $\hat{f}(x,y) = f(x)$ is in $AC(\sigma)$ and satisfies $\|\hat{f}\|_{BV(\sigma)} \le \|f\|_{BV(\sigma_x)}$.
\end{enumerate}
\end{theorem}

When there is no risk of confusion, when $f: \sigma \to \mC$ and $\sigma_0 \subseteq \sigma$ we shall often write $\norm{f}_{BV(\sigma_0)}$ for the norm of the restriction of $f$ to the smaller set. 

An important fact about this sense of absolutely continuity is that it is a `local property'.  We shall say that a set $U$ is a compact neighbourhood of a point $\vecx \in \sigma$ if there exists a bounded open set $V \subseteq \mR^2$ containing $\vecx$ such that $U = \sigma \cap \overline{V}$. 

\begin{theorem}[Patching Lemma {\cite[Theorem~4.1]{DLS1}}]\label{patching-lemma}
Suppose that $\sigma$ is a nonempty compact subset of $\mR^2$ and that $f : \sigma \to \mC$. Then $f \in AC(\sigma)$ if and only if for every point $\vecx \in \sigma$ there exists a compact neighbourhood $U_\vecx$ of $\vecx$ such that $f| U_\vecx \in AC(U_\vecx)$.
\end{theorem}

We note that $AC(\sigma)$ always contains a relatively rich collection of functions. It is shown in \cite[Section 5]{DLS2} that the space of functions which admit a $C^1$ extension to an open neighbourhood of $\sigma$ is always dense in $AC(\sigma)$, as is the space $\CTPP(\sigma)$ of functions which are continuous and triangularly piecewise planar. On the other hand, while Lipschitz functions on a real interval are always absolutely continuous, this result does not extend to arbitrary compact subsets of the plane \cite[Example~3.13]{DLS1}.
 
\section{Classes of sets}

The aim of this paper is to identify the largest class of compact sets in the plane for which we can prove that any absolutely continuous function on such a set can always be extended to be absolutely continuous on any larger compact set.  This class contains all finite unions of polygonal regions, as well as most sets which are finite unions of convex curves. 

Given a set $A \subseteq \mR^2$, we shall denote its closure by $\cl(A)$, its interior by $\intr(A)$ and its boundary by $\partial A$.

\subsection{Polygonal regions}

By a \textbf{polygonal region} we mean a  connected subset of the plane with (simple) polygonal boundary. A set $W$ is a (polygonal) \textbf{window} in a polygonal region $P$ if it is the interior of a polygon $P'$ which lies in $P$. 
Let $\PWP$ denote the collection of all compact subsets of the plane which are finite unions of polygonal regions with finitely many windows removed. 

We allow here that a window may comprise the entire interior of a polygonal region. Note that for a set in $\PWP$, one may have windows inside polygons inside windows inside polygons, and so forth (see Figure~\ref{PWP-sets}). Every set in $\PWP$ is a union of a finite number of polygonal regions and line segments. While not all such sets are in $\PWP$ (for example a single line segment), our main theorem will cover these cases.

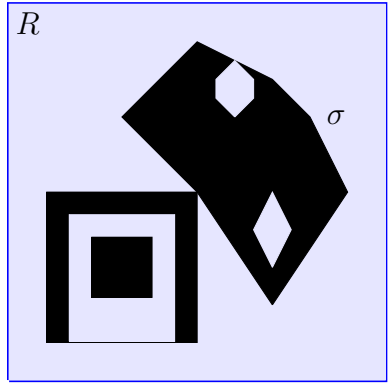
\begin{figure}
\begin{center}
\begin{tikzpicture}[scale=0.5]

\draw[blue,ultra thick] (0,0) -- (10,0) -- (10,10) -- (0,10) -- (0,0);
\draw[fill,blue!10] (0,0) -- (10,0) -- (10,10) -- (0,10) -- (0,0);
\draw[fill,black] (1,1) -- (5,1) -- (5,5) -- (1,5) --(1,1);
\draw[fill,blue!10] (1.6,1.03) -- (4.4,1.03) -- (4.4,4.4) -- (1.6,4.4) -- (1.6,1.03);
\draw[fill,black] (2.2,2.2) -- (3.8,2.2) -- (3.8,3.8) -- (2.2,3.8) -- (2.2,2.2);

\draw[fill,black] (7,2) -- (9,5) -- (8,7) -- (7,8) -- (5,9) -- (3,7) -- (5,5) -- (7,2);
\draw[fill,blue!10] (7,3) -- (7.5,4) -- (7,5) -- (6.5,4) -- (7,3);
\draw[fill,blue!10] (6,7) -- (6.5,7.5) -- (6.5,8) -- (6,8.5) -- (5.5,8) -- (5.5,7.5) -- (6,7);

\draw (0.5,9.5) node {$R$};
\draw (8.7,7) node {$\sigma$};

\end{tikzpicture}
\caption{A set $\sigma \in \PWP$ inside a rectangle $R$. The blue region $R \setminus (\sigma \cup \partial R)$ can always be triangulated. }\label{PWP-sets}
\end{center}
\end{figure}

Let $R$ be a rectangular region containing $\sigma_0 \in \PWP$ in its interior. Then ${\hat \sigma}_0 = \sigma_0 \cup \partial R$ is also in $\PWP$. The region $R \setminus {\hat \sigma}_0$ can always be triangulated. Given $f \in AC(\sigma_0)$ our aim will be to inductively extend $f$ to be absolutely continuous, first on the boundary of each of the triangles, and then on the whole triangular region. This will require that we first show that we can extend from a subset of a convex curve to the whole curve, and then (in Section~\ref{S:fill-poly}) from the boundary of a polygon to its interior.   We shall then show that the final function obtained is absolutely continuous on the whole of $R$. This will require us to show (in Section~\ref{S:joining-thms}) that we can join absolutely continuous functions defined on polygonal regions.

\subsection{Almost polygonally inscribed curves} 

The class of polygonally inscribed curves, or $\PIC$ sets, was introduced in \cite{AS2}. These are compact connected sets which can be written as a finite union of convex curves with certain constraints on how the curves can meet.
For our main theorem we need to relax this condition slightly to deal with the union of a convex curve and a polygonal region. We cannot however completely eliminate all restrictions of how the curves may meet since, as we shall note in Section~\ref{S:apic-join}, a function can be absolutely continuous on each of two convex curves, but not on their union.

\begin{definition}
A (convex) \textbf{polygonal mosaic} in the plane is a finite collection $\calM$ of convex polygonal regions such that 
\begin{enumerate}
  \item $\bigcup_{P \in \calM} P$ is connected;
  \item if $P$ and $Q$ are distinct elements of $\calM$, then $P \cap Q$ is either 
    \begin{itemize}\closeup
      \item empty
      \item a single point which is a vertex of both $P$ and $Q$, or
      \item a line segment.
    \end{itemize}
\end{enumerate}
\end{definition}

\begin{definition} Let $P$  be a convex polygonal region. A set $c \subseteq P$ is a \textbf{convex curve in $P$} if it is a differentiable convex curve joining two vertices of $P$ which only touches the boundary of $P$ at those points.
\end{definition}

\begin{definition}
A nonempty compact connected set $\sigma$ is a \textbf{polygonally inscribed curve} if there exists a polygonal mosaic $\calM = \{P_i\}_{i=1}^M$ whose union contains $\sigma$ and such that for each $i$, $c_i = \sigma \cap P_i$ is a convex curve in $P_i$. The curves $c_i$ will be called the \textbf{components} of $\sigma$. The collection of all polygonally inscribed curves in the plane will be denoted by $\PIC$.
\end{definition}

\begin{definition} 
We shall say that a connected compact set $\sigma \subseteq \mR^2$ is an \textbf{almost polygonally inscribed curve},
written $\sigma \in \APIC$, if there exists a polygonal mosaic $\calM = \{P_i\}_{i=1}^M$ whose union contains $\sigma$ and such that for each $i$, $C_i = \sigma \cap P_i$ is nonempty and is the the union of at most one convex curve in $P_i$ and a (possibly empty) subset of the sides of $P_i$. Given $\sigma \in \APIC$ with a polygonal mosaic $\calM$, the \textbf{components} of $\sigma$ determined by $\calM$ are the convex curves and line segments determined by $\calM$ whose union is $\sigma$.
\end{definition}

Once split into components, every $\PIC$ or $\APIC$ set can be considered as the drawing of a planar graph, with the components giving the edges of the graph, and the intersections of the components the vertices. The graph is not uniquely determined by the set, since such a set can be decomposed into components in many different ways. 

It is worth noting the differences between $\PIC$ sets and $\APIC$ sets. Two components of a $\PIC$ set can meet tangentially, but only with strictly opposite convexity. The new behaviour permitted for an $\APIC$ set is that a curve and a line may meet tangentially at a vertex. Such a vertex will be called a 
\textbf{line tangential vertex}.

\begin{definition} Let $\SPIC$ denote the class of all nonempty compact subsets of $\mR^2$ which are subsets of some almost polygonally inscribed curve.
\end{definition}

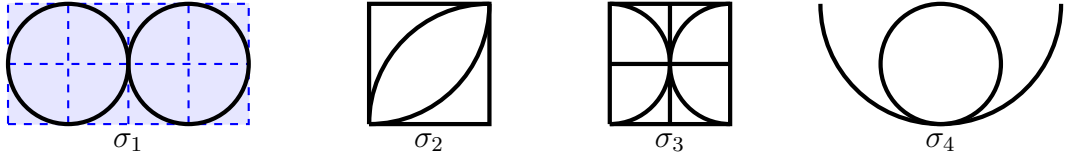
\begin{figure}
\begin{center}
\begin{tikzpicture}[scale=0.8]

\draw[fill,blue!10] (-1,0) -- (3,0) -- (3,2) -- (-1,2) -- (-1,0);
\draw[blue, thick,dashed] (-1,0) -- (3,0) -- (3,2) -- (-1,2) -- (-1,0);
\draw[blue, thick,dashed] (-1,1) -- (3,1);
\draw[blue,thick,dashed] (0,0) -- (0,2);
\draw[blue,thick,dashed] (1,0) -- (1,2);
\draw[blue,thick,dashed] (2,0) -- (2,2);
\draw[ultra thick] (0,1) circle (1cm) ;
\draw[ultra thick] (2,1) circle (1cm);
\draw (1,-0.3) node {$\sigma_1$};

\draw[ultra thick] (5,0) arc (270:360:2cm);
\draw[ultra thick] (7,2) arc (90:180:2cm);
\draw[ultra thick] (5,0) -- (7,0) -- (7,2) -- (5,2) -- (5,0);
\draw (6,-0.3) node {$\sigma_2$};

\draw[ultra thick] (9,0) -- (11,0) -- (11,2) -- (9,2) -- (9,0);
\draw[ultra thick] (10,0) -- (10,2);
\draw[ultra thick] (9,1) -- (11,1);
\draw[ultra thick] (9,0) arc (-90:90:1cm); 
\draw[ultra thick] (11,2) arc (90:270:1cm); 
\draw (10,-0.3) node {$\sigma_3$};

\draw[ultra thick] (14.5,1) circle (1cm);
\draw[ultra thick] (12.5,2) arc (180:360:2cm); 
\draw (14.5,-0.3) node {$\sigma_4$};

\end{tikzpicture}
\caption{$\sigma_1 \in \PIC$, with a polygonal mosaic shown. The sets  $\sigma_2,\sigma_3$ are in $\APIC\setminus \PIC$ since they contain line tangential vertices. The set $\sigma_4$ is not in $\APIC$ due to the curves meeting at a vertex with the same convexity.}\label{APIC}
\end{center}
\end{figure}

The components of sets in $\PIC$ or $\APIC$ are differentiable compact convex curves. A differentiable convex curve $c$ joining $\vecx$ to $\vecy$ is said to be \textbf{projectable} if the orthogonal projection of $c$ onto the line through $\vecx$ and $\vecy$ is just the line segment joining those points. As is shown in \cite{AS2}, any set in $\PIC$ can be broken into projectable component curves. This result can be extended to sets in $\APIC$.

\begin{lemma}\label{split-projectable} Suppose that $\sigma \in \APIC$. Then there exists a polygonal mosaic $\calM = \{P_i\}_{i=1}^M$ for $\sigma$ such that for all $i$, if $\sigma$ contains a convex curve in $P_i$ this curve is projectable.
\end{lemma}

\begin{proof} Let $\calM_0 = \{Q_i\}_{i=1}^{M_0}$ be a polygonal mosaic for $\sigma$. Suppose that $1 \le i \le M_0$ and that $\sigma$ contains a nonprojectable convex curve $c$ in $Q_i$. 
As is shown in \cite[Section 5]{AS2} we can recursively subdivide $Q_i$ into (finitely many) smaller convex polygons so that whenever $c$ passes through ones of the subpolygons, this part of $c$ is projectable (see Figure~\ref{projectable-pic}). We can then replace $Q_i$ in the mosaic with the subpolygons which contain part of $c$, as well as any subpolygons (such as $P_4$ in Figure~\ref{projectable-pic}) with sides containing parts of $\sigma$. Doing this for each such nonprojectable curve will give the required mosaic. 
\end{proof}

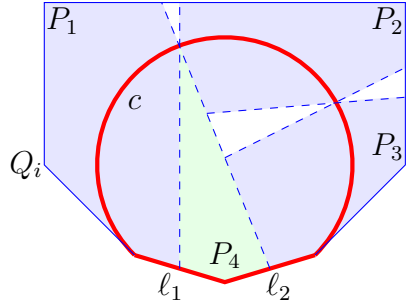
\begin{figure}
\begin{center}
\begin{tikzpicture}[scale=1.2]

\draw[fill,blue!10] (-0.5,-0.15) -- (-0.5,2.3) -- (-0.7,2.8) -- (-2,2.8) -- (-2,1) -- (-1,0) -- (-0.5,-0.15);

\draw[fill,blue!10] (-0.5,2.3) -- (-0.2,1.571) -- (1.25,1.7) -- (2,2.079) -- (2,2.8) -- (-0.5,2.8) -- (-0.5,2.3);
\draw[fill,blue!10] (0.5,-0.15) -- (1,0) -- (2,1) -- (2,1.75) -- (1.25,1.7) -- (0,1.079) -- (0.5,-0.15);
\draw[fill,green!10] (-0.5,-0.15) -- (0,-0.3) -- (0.5,-0.15) -- (-0.5,2.3) -- (-0.5,-0.15);

\draw[red, ultra thick] (1,0) arc (-45:225:1.4142);
\draw[red, ultra thick] (-1,0) -- (0,-0.3) -- (1,0);
\draw[blue] (1,0) -- (2,1) -- (2,2.8) -- (-2,2.8) -- (-2,1) -- (-1,0) ;
\draw[blue,dashed] (-0.5,-0.15) -- (-0.5,2.8);
\draw[blue,dashed] (0.5,-0.15) -- (-0.7,2.8);
\draw[blue,dashed] (0,1.079) -- (2,2.079);
\draw[blue,dashed] (-0.2,1.571) -- (2,1.75);

\draw (-1,1.7) node {$c$};
\draw (-0.6,-0.35) node {$\ell_1$};
\draw (0.6,-0.35) node {$\ell_2$};
\draw (-1.8,2.6) node {$P_1$};
\draw (1.8,2.6) node {$P_2$};
\draw (1.8,1.2) node {$P_3$};
\draw (0,0) node {$P_4$};
\draw (-2.2,1) node {$Q_i$};
\end{tikzpicture}
\caption{Lemma~\ref{split-projectable}: Subdividing a polygon $Q_i$ to produce a finer mosaic containing only projectable convex curves and sides of polygons. Here $\sigma \cap Q_i = c \cup \ell_1 \cup \ell_2$. The polygon $Q_i$ would be replaced here with $P_1,P_2,P_3,P_4$.}\label{projectable-pic}
\end{center}
\end{figure}

Note that any finite union of disjoint $\APIC$ sets is in $\SPIC$ since we can connect the components with line segments. On the other hand, the union of two $\PIC$ sets need not be in $\SPIC$ since one may get two convex curves meeting tangentially with the same convexity (as in $\sigma_4$ in Figure~\ref{APIC}). The following lemma will at least allow us to add line segments to an $\SPIC$ set.

\begin{lemma}\label{add-line-spic}
Suppose that $\sigma_0 \in \SPIC$ and that $\ell$ is a line segment in $\mR^2$. Then $\sigma = \sigma_0 \cup \ell \in \SPIC$.
\end{lemma}

\begin{proof}
By definition, there exists $\tau_0 \in \APIC$ with $\sigma_0 \subseteq \tau_0$. 
Let $\calM = \{P_i\}_{i=1}^n$ be a polygonal mosaic for $\tau_0$ with $c_i$ denoting the convex curve in $P_i$. Without loss we may assume that $\tau_0$ contains all the sides of all the polygons in $\calM$ since $\calM$ is also a polygonal mosaic for this larger set.
Let $M = \cup_{i=1}^n P_i$ and let $\hat \ell$ denote the full line in $\mR^2$ containing $\ell$. 

If $\hat \ell$ does not intersect $M$, then one can certainly connect $\ell$ to $M$ with a line segment $\ell'$, and then find suitable polygons containing $\ell$ and $\ell'$ so that adding these polygons to $\calM$ gives a polygonal mosaic for $\tau = \tau_0 \cup \ell \cup \ell'$. Thus $\tau$ is an $\APIC$ set containing $\sigma$ and hence $\sigma \in \SPIC$.

Otherwise $\hat \ell$ intersects at least one polygon in $\calM$. Let $\ell'$ denote the smallest line segment containing $\ell$ for which $\hat \ell \cap M = \ell' \cap M$, and let $\tau = \tau_0 \cup \ell'$.  Our aim now is to adapt $\calM$ so that it is a polygonal mosaic for $\tau$. 

First choose, if necessary, convex polygons $Q_1,\dots,Q_m$ around any parts of $\ell'$ which are disjoint from $M$. 

Next, suppose that $1 \le i \le n$ and that $\ell_i = \ell' \cap P_i \ne \emptyset$. If $\ell_i$ is a subset of the boundary of $P_i$ there will be nothing to be done as $\ell_i \subset \tau_0$. We may suppose then that $\ell_i$ intersects the interior of $P_i$. 
Let $c_i^\circ = c_i \cap \intr(P_i)$. There are three cases to consider.
\begin{figure}[ht]
	\centering
	\begin{tikzpicture}[scale=2]
		
		\coordinate (v1) at (1,0);
		\coordinate (v2) at (2,0);
		\coordinate (v3) at (3,1);
		\coordinate (v4) at (2,2);
		\coordinate (v5) at (1,2);
		\coordinate (v6) at (0.3,1);
		
		\draw[fill,blue!10] (v1) -- (1.8,0) -- (1.52,0.96) -- (v5) -- (v6) -- (v1);
		\draw[fill,blue!10]  (v2) -- (v3) -- (v4) -- (1.25,2) -- (1.52,0.96) -- (v2);
		
		\draw[blue,thick] (v1) -- (v2) -- (v3) -- (v4) -- (v5) -- (v6) -- (v1);
		\draw[blue,thick,dashed] (v5) -- (v2);
		
		\draw[red,ultra thick] (2,1) arc (90:180:1cm);
		\draw[red,ultra thick] (2,1) -- (v3); 
		\draw[green,thick] (1.8,0) -- (1.25,2);

		\draw (1.55,0.5) node {$\ell_i$};
		\draw (1.2,0.8) node {$c_i$};
		\draw (0.65,1) node {$P_{i,1}$};
		\draw (2.3,1.4) node {$P_{i,2}$};
		
		\coordinate (w1) at (5,0);
		\coordinate (w2) at (6,0);
		\coordinate (w3) at (7,1);
		\coordinate (w4) at (6,2);
		\coordinate (w5) at (5,2);
		\coordinate (w6) at (4.3,1);
		
		\draw[fill,blue!10] (w1) -- (w2) -- (5.29289,0.7071) -- (4.82944,0.243655 ) -- (w1);
		\draw[fill,blue!10]  (w2) -- (w3) -- (6.29289,1.70717) -- (5.29289,0.7071) -- (w2);
		
		\draw[blue,thick] (w1) -- (w2) -- (w3) -- (w4) -- (w5) -- (w6) -- (w1);
		
		\draw[green,ultra thick] (4.82944,0.243655) -- (6.29289,1.70717);
		\draw[red,ultra thick] (6,1) arc (90:180:1cm);
		\draw[red,ultra thick] (6,1) -- (w3); 
		\draw[blue,thick,dashed] (w2) -- (4.65,1.4);
		
		\draw (6,1.6) node {$\ell_i$};
		\draw (6.5,1.1) node {$c_i$};
		\draw (5.5,0.13) node {$P_{i,1}$};
		\draw (6,0.3) node {$P_{i,2}$};
	\end{tikzpicture}
	\caption{The cases where $\ell_i$ meets $c_i$ once.}\label{one-intersection}
	
\end{figure}
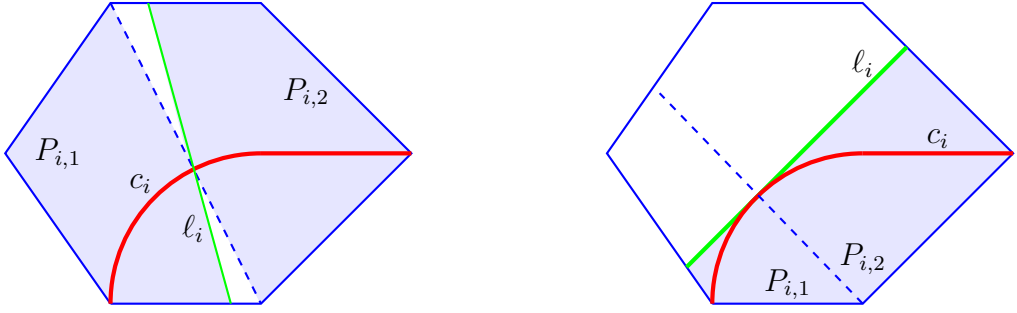

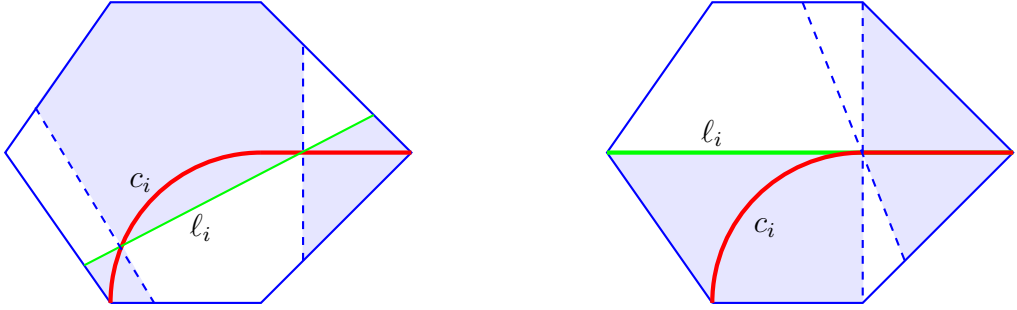
\begin{figure}[ht]
	\centering
	\begin{tikzpicture}[scale=2]
		
		\coordinate (v1) at (1,0);
		\coordinate (v2) at (2,0);
		\coordinate (v3) at (3,1);
		\coordinate (v4) at (2,2);
		\coordinate (v5) at (1,2);
		\coordinate (v6) at (0.3,1);
		
		\draw[fill,blue!10] (0.825,0.25) -- (v1) -- (1.29,0) -- (1.075,0.38) -- (0.825,0.25);
		\draw[fill,blue!10]  (1.075,0.38) -- (2.28,1) -- (2.28,1.72) -- (v4) -- (v5) -- (0.5,1.3);
		\draw[fill,blue!10] (2.28,0.28) -- (v3) -- (2.75,1.25) -- (2.28,1) -- (2.28,0.28);
		
		\draw[blue,thick] (v1) -- (v2) -- (v3) -- (v4) -- (v5) -- (v6) -- (v1);

		\draw[red,ultra thick] (2,1) arc (90:180:1cm);
		\draw[red,ultra thick] (2,1) -- (v3); 
		\draw[green,thick] (0.825,0.25) -- (2.75,1.25);
		\draw[blue,thick,dashed] (0.5,1.3) -- (1.29,0);
		\draw[blue,thick,dashed] (2.28,0.28) -- (2.28,1.72);
		
		\draw (1.6,0.5) node {$\ell_i$};
		\draw (1.2,0.8) node {$c_i$};
		
		\coordinate (w1) at (5,0);
		\coordinate (w2) at (6,0);
		\coordinate (w3) at (7,1);
		\coordinate (w4) at (6,2);
		\coordinate (w5) at (5,2);
		\coordinate (w6) at (4.3,1);
		
		\draw[fill,blue!10] (w1) -- (w2) -- (6,1) -- (w6) -- (w1);
		\draw[fill,blue!10]   (6.28,0.28) -- (w3) -- (w4) -- (6,1) -- (6.28,0.28);
		
		\draw[blue,thick] (w1) -- (w2) -- (w3) -- (w4) -- (w5) -- (w6) -- (w1);
		
		\draw[green,ultra thick] (w6) -- (w3);
		\draw[red,ultra thick] (6,1) arc (90:180:1cm);
		\draw[red,ultra thick] (6,1) -- (w3); 
		
		\draw[blue,thick,dashed] (w4) -- (w2);
		\draw[blue,thick,dashed] (5.6,2) -- (6.28,0.28);
		
		\draw (5,1.13) node {$\ell_i$};
		\draw (5.35,0.5) node {$c_i$};
	\end{tikzpicture}
	\caption{The cases where $\ell_i$ meets $c_i$ more than once. One may also have $\ell_i$ intersecting $c_i$ on a line segment whose endpoints are both in the interior of $P_i$.}\label{many-intersections}
	
\end{figure}

\begin{enumerate}
  \item $\ell_i \cap c_i^\circ = \emptyset$. In this case, $\ell_i$ splits $P_i$ into two convex polygons, and in this case we will replace $P_i$ in $\calM$ with the smaller polygon containing $c_i$.
  \item $\ell_i$ intersects $c_i^\circ$ at a single point $\vecz$. In this case, as in the proof of the Partition Lemma \cite[Lemma 1]{AS2}, we may replace $P_i$ with two smaller convex polygons $P_{i,1}$ and $P_{i,2}$ whose union contains $c_i$, and so that $\ell_i$ lies in the union of the boundaries of these polygons. (See Figure~\ref{one-intersection}.)
   \item $\ell_i$ intersects $c_i^\circ$ in more than one point. In this case we can partition $c_i$ into two or three parts as in Figure~\ref{many-intersections}. 
\end{enumerate}

By making these replacements for each polygon in $\calM$ which meets $\hat \ell$, and by adding in $Q_1,\dots,Q_m$ we will produce a polygonal mosaic for $\tau$, which shows that $\sigma \in \SPIC$. (See Figure~\ref{new-mosaic}.)
\end{proof}

\begin{figure}[ht]
\centering
\begin{tikzpicture}[scale=1.5]

\coordinate (v1) at (0,4);
\coordinate (v2) at (1,3);
\coordinate (v3) at (2,4);
\coordinate (v4) at (3,4);
\coordinate (v5) at (4,3);
\coordinate (v6) at (5,3);
\coordinate (v7) at (5,5);

\draw[fill,blue!7] (v1) -- (v2) --  (1.2929,3.7071) -- (1,4) -- (v1);
\draw[fill,blue!7] (1.2929,3.7071) -- (2,3) -- (3,3) -- (4,3) --  (4,2.8) -- (5.7071,2.8) -- (5.7071,3.2929) --  (6.2,3.2929) -- (6.2,4.7071) --  (5.7071,4.7071) -- (5.7071,5) -- (3,5) -- (v4) -- (2.5,4.5) -- (1.7,4.5) -- (1.2929,3.7071);

\draw[green,line width=0.7mm] (-0.5,4) -- (6.5,4);
\draw[red,ultra thick] (v1) -- (v2);
\draw[red,ultra thick] (v2) arc (0:90:1cm);
\draw[red,ultra thick] (v3) arc (90:180:1cm);
\draw[red,ultra thick] (v3) -- (v4); 
\draw[red,ultra thick] (v5) arc (0:90:1cm);
\draw[red,ultra thick] (v5) -- (v6); 
\draw[red,ultra thick] (v6) arc (-90:90:1cm);
\draw[red,ultra thick] (v7) -- (v4); 

\draw[blue,thick,dashed] (v1) -- (v2) -- (1,4) -- (v1);
\draw[blue,thick,dashed] (v2) -- (1.2929,3.7071) -- (1,4);
\draw[blue,thick,dashed] (1.2929,3.7071) -- (2,3) -- (3,3) -- (v4) -- (2.5,4.5) -- (1.7,4.5) -- (1.2929,3.7071);
\draw[blue,thick,dashed]  (3,3) -- (v5) -- (5,4) -- (v4);
\draw[blue,thick,dashed]   (v5) -- (4,2.8) -- (5.7071,2.8) -- (5.7071,3.2929) -- (5,4);
\draw[blue,thick,dashed]  (5,4) -- (v7) -- (3,5) -- (v4);
\draw[blue,thick,dashed]   (5.7071,3.2929) -- (6.2,3.2929) -- (6.2,4.7071) --  (5.7071,4.7071) -- (5,4);
\draw[blue,thick,dashed]  (v7) -- (5.7071,5) -- (5.7071,4.7071);
 
\draw (-0.2,4.13) node {$\ell$};
\draw (4,4.3) node {$\sigma_0$};

\coordinate (w1) at (0,1);
\coordinate (w2) at (1,0);
\coordinate (w3) at (2,1);
\coordinate (w4) at (3,1);
\coordinate (w5) at (4,0);
\coordinate (w6) at (5,0);
\coordinate (w7) at (5,2);

\draw[fill,green!10] (-0.5,1) -- (-0.25,0.85) -- (w1) -- (-0.25,1.15) -- (-0.5,1);
\draw[fill,green!10] (1.43,1) -- (1.21,1.15) -- (1,1) -- (1.2929,0.7071) -- (1.43,1);

\draw[fill,blue!7]  (2,0) -- (2,1) -- (1.43,1) -- (1.2929,0.7071) -- (2,0);
\draw[fill,blue!7] (w3) -- (2.3,0) -- (3,0) -- (w4) -- (2.5,1.5) -- (2.0,1.5) -- (w3); 
\draw[fill,blue!7] (5,1) -- (6,1) -- (6.2,1.13333) -- (6.2,1.7071) -- (5.7071,1.7071) -- (5,1);
\draw[fill,blue!7] (5.4,0.6) --  (5.7071,0.2929) --  (6.2,0.2929) -- (6.2,1) -- (6,1) -- (5.4,0.6);
\draw[fill,blue!7] (w1) -- (w2) -- (1,1) -- (w1);
\draw[fill,blue!7] (w2) -- (1.2929,0.7071) -- (1,1);
\draw[fill,blue!7] (3,0) -- (w5) -- (4,-0.2) -- (5.7071,-0.2) -- (5.7071,0.2929) -- (5,1) -- (5.7071,1.7071)  --   (5.7071,2) -- (3,2) -- (3,0);

\draw[fill,green!10] (6.2,1) -- (6.35,0.85) -- (6.5,1) -- (6.35,1.15) -- (6.2,1);

\draw[red,ultra thick] (-0.5,1) -- (6.5,1);
\draw[red,ultra thick] (w1) -- (w2);
\draw[red,ultra thick] (w2) arc (0:90:1cm);
\draw[red,ultra thick] (w3) arc (90:180:1cm);
\draw[red,ultra thick] (w3) -- (w4); 
\draw[red,ultra thick] (w5) arc (0:90:1cm);
\draw[red,ultra thick] (w5) -- (w6); 
\draw[red,ultra thick] (w6) arc (-90:90:1cm);
\draw[red,ultra thick] (w7) -- (w4); 

\draw[blue,thick,dashed] (-0.5,1) -- (-0.25,0.85) -- (w1) -- (-0.25,1.15) -- (-0.5,1);
\draw[blue,thick,dashed] (w1) -- (w2) -- (1,1) -- (w1);
\draw[blue,thick,dashed] (w2) -- (1.2929,0.7071) -- (1,1);
\draw[blue,thick,dashed] (1.2929,0.7071) -- (2,0) -- (3,0) -- (w4) -- (2.5,1.5) -- (1.7,1.5) -- (1.2929,0.7071);
\draw[blue,thick,dashed] (1.2929,0.7071) -- (2,0) -- (2,1) -- (1.43,1);
\draw[blue,thick,dashed]  (2.0,1.5) -- (w3) -- (2.3,0) -- (3,0); 
\draw[blue,thick,dashed] (1.43,1) -- (1.21,1.15) -- (1,1);
\draw[blue,thick,dashed] (5.4,0.6) -- (6.2,1.13333);

\draw[blue,thick,dashed]  (3,0) -- (w5) -- (5,1) -- (w4);
\draw[blue,thick,dashed]   (w5) -- (4,-0.2) -- (5.7071,-0.2) -- (5.7071,0.2929) -- (5,1);
\draw[blue,thick,dashed]  (5,1) -- (w7) -- (3,2) -- (w4);
\draw[blue,thick,dashed]   (5.7071,0.2929) -- (6.2,0.2929) -- (6.2,1.7071) --  (5.7071,1.7071) -- (5,1);
\draw[blue,thick,dashed]  (w7) -- (5.7071,2) -- (5.7071,1.7071);
\draw[blue,thick,dashed]  (6.2,1) -- (6.35,0.85) -- (6.5,1) -- (6.35,1.15) -- (6.2,1);
\draw (4,1.35) node {$\sigma$};

\end{tikzpicture}
\caption{Lemma~\ref{add-line-spic}: adapting a mosaic for the addition of an extra line segment. The top diagram shows a line segment $\ell$ (in green) and a $\SPIC$ set $\sigma_0$ (in red). A polygonal mosaic for $\sigma_0$ is shown by the dashed polygons. \hfill\break
The lower diagram shows the addition of three new polygons (in green) and splitting of two polygons into smaller pieces to form a polygonal mosaic for $\sigma = \sigma_0 \cup \ell$.}\label{new-mosaic}

\end{figure}
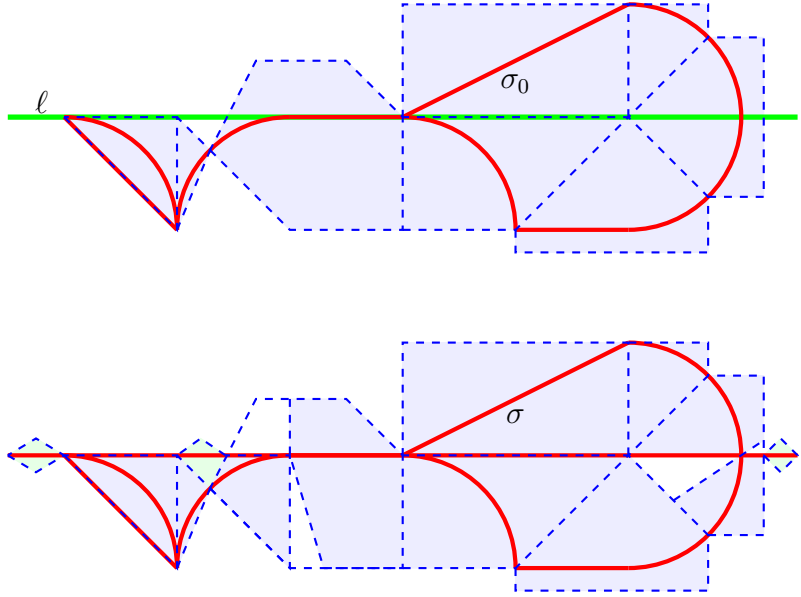


\section{Extending from the real line}

The first, and rather elementary, step in the proof of the main theorem is to show that if $\sigma_0$ is any compact subset of the real line, and $f \in AC(\sigma_0)$, then $f$ can be extended  to any larger compact subset.

Suppose then that $\sigma_0 \subseteq \sigma \subseteq \mR$. Let $a_0 = \min \sigma_0$ and let $b_0 = \max \sigma_0$. Let $J = [a,b]$ denote a compact interval containing $\sigma$ in its interior. Given $f \in AC(\sigma_0)$ define the extension $\iota(f): J \to \mC$ by making it constant on $[a,a_0]$ and $[b_0,b]$ and piecewise linear on the open intervals of $[a_0,b_0] \setminus \sigma_0$. 

\begin{lemma}\cite[Corollary 2.14]{AD1}\label{iota-lemma}
The map $\iota$ is a linear isometry from $AC(\sigma_0)$ to $AC(J)$. 
\end{lemma}

Extending from a subset of the line to a larger subset of the plane is straightforward. By affine invariance, the following result can be applied to any function which is constant on a family of parallel lines.

\begin{theorem}\label{real-ext} Suppose that $\sigma_0$ is a nonempty compact subset of $\mR$ and that $\sigma \subseteq \mR^2$ is a compact superset of $\sigma_0$.  If $f \in AC(\sigma_0)$ then there exists an extension $\hatf \in AC(\sigma)$ with $\ssnorm{\hatf}_{BV(\sigma)} = \norm{f}_{BV(\sigma_0)}$.
\end{theorem}

\begin{proof}
Combining Lemma~\ref{iota-lemma} and Theorem~\ref{ac-properties}\ref{r-ext} gives an extension $\hatf$ to $\sigma$ with norm at most $\norm{f}_{BV(\sigma_0)}$. That the norms must be equal follows from Theorem~\ref{ac-properties}\ref{restr}.
\end{proof}


\section{Convex curves}

The results for subsets of the real line can be extended to subsets of curves as long as these curves are well enough behaved. The following theorem contains the content of Section~8 of \cite{AS2}. (We remind the reader that a projectable curve is differentiable except at its endpoints.)

\begin{theorem}\label{c-isom-01}
Suppose that $c$ is a projectable convex curve in the plane. Then there exists a Banach algebra isomorphism $\Phi: AC(c) \to AC[0,1]$ with $\norm{\Phi} \le 2$ and $\norm{\Phi^{-1}} =1$.
\end{theorem}

More specifically, if $h: [0,1] \to c$ is a differentiable parameterization of $c$, then we can take the isomorphism to be $\Phi_h(f) = f \circ h$. 

\begin{lemma}\label{subsets-of-convex}
Suppose that $c$ is a projectable convex curve parameterized by a differentiable function $h: [0,1] \to c$, and that $\emptyset \ne \sigma_0 \subseteq c$. Let $\tau_0 =  h^{-1}(\sigma_0) \subseteq [0,1]$. Then $\Phi_0(f) = f \circ h$ defines an isomorphism between $AC(\sigma_0)$ and $AC(\tau_0)$.
\end{lemma}

\begin{proof}
First note that since $\Phi_h(f) = f \circ h$ defines an isomorphism between $AC(c)$ and $AC[0,1]$, \cite[Theorem~4.7]{AS3} implies that $h \in AC[0,1]$ and $h^{-1} \in AC(c)$. But by Theorem~\ref{ac-properties}\ref{restr}, this means that $h|\tau_0 \in AC(\tau_0)$ and $h^{-1}|\sigma_0 \in AC(\sigma_0)$. It is an easy consequence of \cite[Theorem~3.11]{AS3} that $\Phi_0$ is a Banach algebra isomorphism from $BV(\sigma_0)$ to $BV(\tau_0)$. Applying \cite[Theorem~4.7]{AS3} again then, we can deduce that $\Phi_0$ acts as an isomorphism from $AC(\sigma_0)$ to $AC(\tau_0)$.
\end{proof} 

\begin{theorem}\label{conv-curve-extend}
Suppose that $c$ is a projectable convex curve, and that $\emptyset \ne \sigma_0 \subseteq c$. If $f \in AC(\sigma_0)$ then there exists an extension $\hat f \in AC(c)$ with $\ssnorm{\hat f}_{BV(c)} \le 2\norm{f}_{BV(\sigma_0)}$
\end{theorem}

\begin{proof}
Suppose that $f \in AC(\sigma_0)$. Let $\Phi_0$ and $h$ be as in Lemma~\ref{subsets-of-convex}. Let $g = \Phi_0(f) \in AC(h^{-1}(\sigma_0))$. Lemma~\ref{iota-lemma} allows us to extend $g$ to ${\hat g} \in AC[0,1]$, and so ${\hat f} = \Phi_h^{-1}(\hat g) \in AC(c)$. The norm bounds follow from the properties of $\Phi$.
\end{proof}

One can extend Theorem~\ref{real-ext} to cover the case where $\sigma_0$ is a nonempty compact subset of a projectable curve. For the moment however, we shall just give a special case which we shall need in the later sections. In this result, the curve $c$ need not be projectable.

\begin{lemma}\label{from-c-to-T}
Suppose that $T$ is a triangle in $\mR^2$ with vertices at $\vecx$, $\vecy$ and $\vecz$. Let $c$ be a differentiable convex curve joining $\vecx$ to $\vecy$. Then every $f \in AC(c)$ admits an extension $\hat f \in AC(T)$ with $\ssnorm{\hat f}_{BV(T)} \le 2 \norm{f}_{BV(c)}$. 
\end{lemma} 

\begin{proof}
Let $\vecx' = (0,0)$, $\vecy' = (1,0)$ and $\vecz' = (\frac{1}{2},1)$ and let $h$ be the affine transformation which maps $\vecx$ to $\vecx'$, $\vecy$ to $\vecy'$ and $\vecz$ to $\vecz'$. Let $T' = h(T)$ and let $c' = h(c)$.  Then $c'$ is a projectable convex curve which we may assume is the graph of the convex function $k:[0,1] \to \mR$. As above, we may define $g \in AC[0,1]$ by $g(t) = f(h^{-1}(t,k(t)))$, and then setting ${\hat g}(x,y) = g(x)$  gives an extension of $g$ to all of $T'$. Now setting $\hatf = {\hat g} \circ h$ gives an absolutely continuous extension of $f$ to $T$. The bounds follow from Theorem~\ref{ac-properties}\ref{aff-inv} and Theorem~\ref{real-ext}.
\end{proof}


\section{Filling in polygons}\label{S:fill-poly}

The other main building block of the sets to which the main theorem will apply are polygonal regions of the plane. The first step is to show that one may always extend an absolutely continuous function defined on the boundary of such a set into its interior. 

In the case where the polygonal region is the unit square $S = [0,1] \times [0,1]$, one can write an explicit formula for an extension. This result, with a slightly different bound, is \cite[Theorem~3.5]{AD3}.

\begin{theorem}\label{extend-S} Suppose that $\sigma_0$ is the boundary of the unit square $\sigma = S$ and that $f \in AC(\sigma_0)$. Then $f$ admits an extension $\hatf \in AC(\sigma)$ with
  \[ \ssnorm{\hatf}_{BV(\sigma)} \le 10 \norm{f}_{AC(\sigma_0)}. \]
\end{theorem}

\newcommand{\rspace}{\hphantom{b(x)b(x)+\ell(y) -b(0}}

\begin{proof} The restrictions of $f$ to the four sides of $S$ are absolutely continuous:
  \[ b(x) = f(x,0), \qquad t(x) = f(x,1), \qquad \ell(y) = f(0,y), \qquad r(y) = f(1,y). \]
For $(x,y) \in S$ let
  \begin{align}\label{fhat1}
    \hatf(x,y) &= b(x)+\ell(y) -b(0) +(b(0)-b(1)+r(y)-\ell(y))x   \notag \\
                 & \rspace + (\ell(0)-\ell(1)+t(x)-b(x))y \notag\\
                 & \rspace + (b(1)-b(0)+t(0)-t(1))xy .
  \end{align}
It is readily verified that $\hatf$ is indeed an extension of $f$. Further, using Theorem~\ref{real-ext}, it is clear that $\hatf \in AC(\sigma)$. 

Note that $\hatf$ can also be written as
  \begin{align}\label{fhat2}
    \hatf(x,y) &= b(0) (x-1)(1-y) + b(1) x (1-y) + t(0) (x-1)y + t(1) xy \notag\\
                 &\rspace + b(x)(1-y) + \ell(y)(1-x)+r(y) x + t(x) y
  \end{align}
and so $\ssnorm{\hatf}_\infty \le 8 \norm{f}_\infty$. 
Using Lemma~2.7 and Theorem 2.16 from \cite{DLS1},
\begin{align*}
 \var((b(0)-b(1))x,\sigma) & \le \var(b,[0,1]) \var(x,\sigma) = \var(b,[0,1]) \\
 \var(r(y)x,\sigma) & \le \var(r(y),\sigma) \norm{x}_\infty + \norm{r}_\infty \var(x,\sigma) = \norm{r}_{BV[0,1]} 
\end{align*}
while
\begin{align*}
 \var((b(1)-b(0) &+t(0)-t(1))xy,\sigma) \le (\var(b,[0,1])+\var(t,[0,1]) \var(xy,\sigma) \\
            & \le (\var(b,[0,1])+\var(t,[0,1])
                (\norm{x}_\infty \var(y,\sigma) + \var(x,\sigma) \norm{y}_\infty \\
            &= 2(\var(b,[0,1])+\var(t,[0,1]).
\end{align*}
Using similar estimates for the other terms in (\ref{fhat1}), this shows that
  \begin{align*}
   \var(\hatf,\sigma) & \le \var(b,[0,1]) + \var(\ell,[0,1]) + \norm{r}_{BV[0,1]} 
                           + \norm{\ell}_{BV[0,1]} \\
                     & \qquad + \var(\ell,[0,1]) + \norm{t}_{BV[0,1]} 
                           + \norm{b}_{BV[0,1]} \\
                     & \qquad + 2(\var(b,[0,1])+\var(t,[0,1]) \\
                     & \le 4 \var(f,\sigma_0)  + \norm{b}_{BV[0,1]} + \norm{t}_{BV[0,1]} 
                         + \norm{\ell}_{BV[0,1]} + \norm{r}_{BV[0,1]}.
  \end{align*}

In the proof of Theorem 3.4 of \cite{AD3} it was shown that
  \[ \norm{b}_{BV[0,1]} + \norm{t}_{BV[0,1]} + \norm{\ell}_{BV[0,1]} + \norm{r}_{BV[0,1]}  \le 2 \norm{f}_{BV(\sigma_0)}. \]
  
Combining all of this gives
  \begin{align*} \ssnorm{\hatf}_{BV(\sigma)}
     &= \ssnorm{\hatf}_\infty + \var(\hatf,\sigma) \\
     &\le 8 \norm{f}_\infty + 4 \var(f,\sigma_0) + 2 \norm{f}_{BV(\sigma_0)}
     \le 10 \norm{f}_{BV(\sigma_0)}.  \qedhere
  \end{align*}
\end{proof}

The second author \cite{St} obtained a slightly better bound for an extension from the boundary of a triangle to its interior. As the construction is a little more complicated than the one for a square, we just record the result here.

\begin{theorem} \cite[Theorem~6.3.6]{St} Suppose that $\sigma_0$ is the boundary of a triangular region $\sigma$ in the plane. If $f \in AC(\sigma_0)$ then $f$ admits an extension $\hatf \in AC(\sigma)$ with
  \[ \ssnorm{\hatf}_{BV(\sigma)} \le 7 \norm{f}_{AC(\sigma_0)}. \]
\end{theorem}

It is likely that the constants in both cases for the bounds on the extensions are not sharp, even for the explicit extensions constructed. It is known that the construction in the proof of Theorem~\ref{extend-S} may produce an extension with a larger norm than the initial function.

\begin{example}
Consider the following example of the construction from Theorem~\ref{extend-S}. If $b(x) = t(x) = x(1-x)$ and $\ell(y) = r(y) = y(1-y)$, then $\hatf(x,y) = x(1-x)+y(1-y)$. It is easy to check that 
$\ssnorm{\hatf}_{BV(\sigma)} = 6$. Now, as in the proof of Section~2.6 of \cite{DLS1}, $\var(f,\sigma_0)$ is at most twice the parameterized variation of $f$ around the edges of the square, so 
$\norm{f}_{BV(\sigma_0)} \le 1 + 2 = 3$. 
\end{example}

Using the results of \cite{DL} (see Theorem 6.3), one can obtain extensions to the region bounded by any polygon.

\begin{theorem}\label{isom-poly} Suppose that $\sigma$ is a polygonal region in $\mR^2$ with boundary $\sigma_0$. Then there exists a homeomorphism $h: \mR^2 \to \mR^2$, made up of the composition of locally piecewise affine maps, such that
  \begin{itemize}\closeup
  \item $h(\sigma)$ is the square $S = [0,1] \times [0,1]$,
  \item $AC(\sigma_0)$ is isomorphic to $AC(\partial S)$,
  \item $AC(\sigma)$ is isomorphic to $AC(S)$.
  \end{itemize}
In each case the algebra isomorphism is given by $\Phi(g) = g \circ h^{-1}$.
\end{theorem}

\begin{theorem}\label{fill-poly}
Suppose that $\sigma$ is a polygonal region in $\mR^2$ with boundary $\sigma_0$, and that $f \in AC(\sigma_0)$. Then $f$ admits an extension $\hatf \in AC(\sigma)$ with
  \[ \ssnorm{\hatf}_{BV(\sigma)} \le K_\sigma \norm{f}_{BV(\sigma_0)} \]
where $K_\sigma$ only depends on $\sigma$. 
\end{theorem}

\begin{proof}
Let $h$ and $\phi$ be the maps guaranteed by Theorem~\ref{isom-poly}. Then $g = \Phi(f) \in AC(\partial S)$ and so by Theorem~\ref{extend-S}, $g$ admits an extension ${\hat g} \in AC(S)$. The form of $\Phi$ then ensures that $\Phi^{-1}({\hat g})$ is an extension of $f$. 

Finally $\ssnorm{\hatf} \le \norm{\Phi^{-1}} \norm{\hat g} \le 18 \norm{\Phi^{-1}} \norm{\Phi} \norm{f}$, so we can take $K_\sigma = 18 \norm{\Phi^{-1}} \norm{\Phi}$ which only depends on $\sigma$. 
\end{proof}

Examining the proof of Theorem~\ref{isom-poly} in \cite{DL} shows that the norm $K_\sigma$ in the above proof actually only depends on the number of sides of $\sigma$. 


\section{Joining theorems}\label{S:joining-thms}

Many of the later constructions will require that an extension be separately constructed on different parts of the larger set $\sigma$. The challenge then is to show that the function is absolutely continuous on the whole of the set. In general, if $\sigma = \sigma_1 \cup \sigma_2$ then knowing that $f|\sigma_1 \in AC(\sigma_1)$ and $f|\sigma_2 \in AC(\sigma_2)$ is not enough to deduce that $f \in AC(\sigma)$. 

\begin{example}\label{bad-join} Suppose that $\{x_k\}_{k=1}^\infty$ is a strictly decreasing sequence of positive numbers with limit $0$.
Let $\sigma = \{ x_k \st k = 1,2,3,\dots\bigr\} \cup \{0\}$, $\sigma_1 = \{x_{2k} \st k = 1,2,3,\dots\} \cup\{0\}$ and $\sigma_2 = \{x_{2k-1} \st k = 1,2,3,\dots\} \cup \{0\}$. Define $f: \sigma \to \mC$ by $f(x_k) = \frac{(-1)^k}{k}$, with $f(0) = 0$. Then the restrictions of $f$ to $\sigma_1$ and to $\sigma_2$ are absolutely continuous, but $f$ is not even of bounded variation on $\sigma$. 
\end{example}

If $f_1: \sigma_1 \to \mC$ and $f_2: \sigma_2 \to \mC$ agree on $\sigma_1 \cap \sigma_2$, we shall call the function determined by 
  $ f(\vecx) = f_i(\vecx)$ for $\vecx \in \sigma_i$ the \textbf{joined function} of $f_1$ and $f_2$.
If $\sigma_1$ and $\sigma_2$ are disjoint, then the Patching Lemma (Theorem~\ref{patching-lemma}) implies that $f \in AC(\sigma)$. The difficulty then occurs in the case that $\sigma_1$ and $\sigma_2$ overlap. Under suitable conditions on $\sigma_1$ and $\sigma_2$ one can at least rule out the possibility that the joined function fails to be of bounded variation. 

We shall say that $\sigma_1$ and $\sigma_2$ \textbf{join convexly} if for given any $\vecx \in \sigma_1 \setminus \sigma_2$ and $\vecy \in \sigma_2 \setminus \sigma_1$ the line segment joining $\vecx$ and $\vecy$ contains a point $\vecw \in \sigma_1 \cap \sigma_2$. 

\begin{theorem}\label{join-conv} {\rm \cite[Theorem 3.8]{DLS1}} 
Suppose that $\sigma_1,\sigma_2 \subseteq \mR^2$ are nonempty compact sets which are disjoint except at their boundaries, and that $\sigma_1$ and $\sigma_2$ join convexly. Let $\sigma = \sigma_1 \cup \sigma_2$. If $f_1 \in BV(\sigma_1)$ and $f_2 \in BV(\sigma_2)$ agree on $\sigma_1 \cap \sigma_2$, then their joined function $f$ lies in $BV(\sigma)$ and $\norm{f}_{BV(\sigma)} \le \norm{f_1}_{BV(\sigma_1)} + \norm{f_2}_{BV(\sigma_2)}$.
\end{theorem}

If $\sigma_1$ and $\sigma_2$ lie on either side of a straight line  then one can, at the cost of an extra factor, remove the convexity hypothesis. Let $H^+$ and $H^-$ denote the closed upper and lower half-planes in $\mR^2$ and let $L$ be the $x$-axis.

\begin{theorem}\label{join-line}
Suppose that $\sigma_1 \subseteq H^+$ and $\sigma_2 \subseteq H^-$ are nonempty compact sets, that $\sigma_2 \cap L \subseteq \sigma_1 \cap L$, and that $\sigma = \sigma_1 \cup \sigma_2$.
If $f: \sigma \to \mC$ then
  \[ \norm{f}_{BV(\sigma)} \le 2\bigl( \norm{f|\sigma_1}_{BV(\sigma_1)} + \norm{f|\sigma_2}_{BV(\sigma_2)} \bigr). \]
\end{theorem}
  
\begin{proof}
Suppose that $S = [\vecx_0,\vecx_1,\dots,\vecx_n]$ is a list of elements of $\sigma$ and let $J = \{1,\dots,n\}$.
Let
\begin{align*}
  J_1 &= \{j \in J \st \vecx_{j-1},\vecx_j \in \sigma_1\}, \\
  J_2 &= \{j  \in J \st \vecx_{j-1},\vecx_j \in \sigma_2 \}, \\
  J_3 &= J \setminus (J_1 \cup J_2).
\end{align*}
Then, noting that $J_1$ and $J_2$ need not be disjoint, and considering empty sums as zero,
  \[
  \cvar(f,S) = \sum_{j=1}^n |f(\vecx_j) - f(\vecx_{j-1})| \\
               \le \sum_{k=1}^3 \sum_{j \in J_k} |f(\vecx_j) - f(\vecx_{j-1})|.
  \]
As in the proof of  \cite[Lemma~4.2]{DLS1}, $ \sum_{j \in J_1} |f(\vecx_j) - f(\vecx_{j-1})| \le \var(f,\sigma_1) \vf(S)$ and $ \sum_{j \in J_2} |f(\vecx_j) - f(\vecx_{j-1})| \le \var(f,\sigma_2) \vf(S)$. 
Note that if $j \in J_3$ then one end of the line segment $s_j = \ls[\vecx_{j-1},\vecx_j]$ lies $H^+$ while the other end, which we shall denote $\vecu_j = (u_j,v_j)$, lies in the open lower half-plane. Thus $v = \sup_{j \in J_3} v_j$ is strictly negative and each of these line segments crosses the horizontal line $y = v/2$. This implies that $\vf(S)$ is at least the number of elements in $J_3$. Thus
  \[  \sum_{j \in J_3} |f(\vecx_j) - f(\vecx_{j-1})|
   \le \bigl(\norm{f|\sigma_1}_\infty + \norm{f|\sigma_2}_\infty\bigr) \vf(S). \] 
It follows that 
  \[ \frac{\cvar(f,S)}{\vf(S)} \le \norm{f|\sigma_1}_{BV(\sigma_1)} + \norm{f|\sigma_2}_{BV(\sigma_2)}, \]
and so, on taking the supremum over all such lists $S$,
  \[ \norm{f}_{BV(\sigma)} = \norm{f}_\infty + \var(f,\sigma)
                      \le 2 \bigl( \norm{f|\sigma_1}_{BV(\sigma_1)} + \norm{f|\sigma_2}_{BV(\sigma_2)} \bigr). \qedhere \]
\end{proof}

By considering the case where $\sigma_1$ and $\sigma_2$ are disjoint and $f$ is the characteristic function of $\sigma_1$, one can observe that the factor 2 in the above theorem is necessary.
Of course by affine invariance, the same result holds with the $x$-axis replaced by any other line in the plane.

An important special case is when a line and a convex curve meet at a line tangential vertex.


\begin{corollary}\label{LT-join-bound}
Suppose that the line segment $\ell = \ls[\vecx,\vecz]$ meets the  convex curve $c$ tangentially at $\vecx$, and let $\sigma = \ell \cup c$. If $f: \sigma \to \mC$ then 
  \[ \norm{f}_{BV(\sigma)} \le 2 \bigl(\norm{f}_{BV(\ell)}+\norm{f}_{BV(c)} \bigr). \]
\end{corollary}

\begin{proof}
After applying an appropriate affine transformation one need consider the case where $\ell = [0,1]$ and $c$ lies in the closed lower half-plane. The result then follows from Theorem~\ref{join-line} with $\sigma_1 = \ell$ and $\sigma_2 = c$.
\end{proof}

For the later results we shall piece together functions defined on polygonal pieces.  
It is not hard to adapt Example~\ref{bad-join} to show that even in the case of two polygonal sets $\sigma_1$ and $\sigma_2$, there is no constant $C$ such that $\norm{f}_{BV(\sigma)} \le C(\norm{f}_{BV(\sigma_1)} + \norm{f}_{BV(\sigma_2)})$. The next result will however allow us some control over the norms if we add triangular pieces one at a time.

Let $R_1$ and $R_2$ be two distinct rays in $\mR^2$ starting at the origin and let $S_1$ and $S_2$ be the two closed sectors of the plane with boundary $B = R_1 \cup R_2$. A half-plane affine map will refer to a homeomorphism of the plane which is affine on each of a pair of complementary half-planes. (Further details can be found in \cite[Section~4]{DL}.) 

\begin{theorem}\label{join-sectors}
Suppose that $\sigma_1 \subseteq S_1$ and $\sigma_2 \subseteq S_2$ are nonempty compact subsets of complementary sectors with boundary $B$, $\sigma_1 \cap B = \sigma_2 \cap B$, and $\sigma = \sigma_1 \cup \sigma_2$.
If $f: \sigma \to \mC$, then
  \[ \norm{f}_{BV(\sigma)} \le 8 \bigl( \norm{f|\sigma_1}_{BV(\sigma_1)} + \norm{f|\sigma_2}_{BV(\sigma_2)} \bigr). \]
\end{theorem}

\begin{proof} There exists a half-plane affine map $\alpha$ which maps $B$ onto a straight line through the origin. Let ${\hat \sigma}_1$, ${\hat \sigma}_2$ and $\hat \sigma$ be the images of $\sigma_1$, $\sigma_2$ and $\sigma$ under $\alpha$, and let 
$g: {\hat \sigma} \to \mC$ be $g = f \circ \alpha^{-1}$. Then by \cite[Theorem~4.5]{DL} and Theorem~\ref{join-line}
  \begin{align*}
  \norm{f}_{BV(\sigma)} & \le 2 \ssnorm{g}_{BV({\hat \sigma})} \\
                     &\le 4 \bigl( \ssnorm{g}_{BV({\hat \sigma}_1)} 
                              + \ssnorm{g}_{BV({\hat \sigma}_2)} \bigr) \\
                     &\le 8 \bigl( \norm{f}_{BV({\sigma}_1)} 
                              + \norm{f}_{BV({\sigma}_2)} \bigr). \qedhere
  \end{align*}
\end{proof} 

The factor $8$ here is unlikely to be sharp, but again by considering characteristic functions, one can see that one needs at least a factor of $3$ in the bound.

\begin{theorem}\label{pasting-lemma-ac}
Suppose that $\sigma_1 \subseteq H^+$ and $\sigma_2 \subseteq H^-$ are nonempty compact sets, that $\sigma_2 \cap L \subseteq \sigma_1 \cap L$, and that $\sigma = \sigma_1 \cup \sigma_2$. 
Suppose
that $f: \sigma \to \mC$. If
$f|\sigma_1 \in \AC(\sigma_1)$ and $f|\sigma_2 \in \AC(\sigma_2)$, then $f \in
\AC(\sigma)$ 
 and
  \[ \normbv{f}\le 2\bigl( \norm{f|\sigma_1}_{\BV(\sigma_1)} +
\norm{f|\sigma_2}_{\BV(\sigma_2)} \bigr). \]
\end{theorem}

\begin{figure}[ht!]
\begin{center}
\begin{tikzpicture}[scale=1]

   \begin{scope}
      \clip (0,2) rectangle (6,4);
      \fill[blue!20,rotate around={45:(3,2)}] (3,2) ellipse (2cm and 1cm);
   \end{scope}
   \begin{scope}
      \clip (0,0) rectangle (6,2);
      \fill[green!20,rotate around={45:(3,2)}] (3,2) ellipse (2cm and 1cm);
   \end{scope}
     \draw[thick, rotate around={45:(3,2)}] (3,2) ellipse (2cm and 1cm);
     \draw[white,fill] (2.5,3.7) -- (3.5,3) -- (4.5,3.7) -- (2.5,3.7);
     \draw[black,thick] (3.02,3.3) -- (3.5,3) -- (4.25,3.52);
     \draw[white,fill] (3,1.3) circle (0.3cm);
     \draw[black,thick] (3,1.3) circle (0.3cm);
     \draw (0,2) -- (6,2);
    \draw (3.2,2.3) node[right] {$\sigma_1$};
    \draw (2,1.6) node[right] {$\sigma_2$};
    \draw[red,dashed] (1,0) -- (5.4,0) -- (5.4,4) -- (1,4) -- (1,0) ;
    \draw[red] (2,3.7) node {$P$};
\end{tikzpicture}
\end{center}
\caption{Diagram for Theorem \ref{pasting-lemma-ac}.}
\label{join-ellipse}
\end{figure}
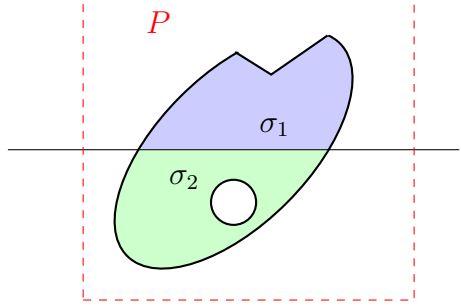

\begin{proof} The bound on the variation norm comes from the previous theorem, so the task is to show that $f$ is absolutely continuous.
If $\sigma_1 \cap \sigma_2 = \emptyset$, then the absolutely continuity follows immediately from Theorem~\ref{patching-lemma}

Suppose then that $\sigma_1 \cap \sigma_2 \ne \emptyset$ so that $\tau = \sigma_1 \cap L$ is nonempty. 
Choose $a_0,b_0 \in \mR$ such that $a_0 \le x \le b_0$ for all $(x,y) \in \sigma$ and let $\tau_0 = [a_0,b_0] \times \{0\}$. Let $P = [a_0,b_0] \times [c_0,d_0]$ be a closed rectangle containing $\sigma$. Let $P_1 = [a_0,b_0] \times [0,d_0]$ and let $P_2 = [a_0,b_0] \times [c_0,d_0]$, so that $\sigma_1 \subseteq P_1$ and $\sigma_2 \subseteq P_2$.

Now $f|\tau \in AC(\tau)$ so by Lemma~\ref{iota-lemma} and Theorem~\ref{ac-properties}\ref{r-ext} there is a function $g \in AC(\tau_0)$ such that $g|\tau = f|\tau$.  By setting $g(x,y) = g(x,0)$ for all $y$, we can then extend to $g \in AC(P)$ with $\norm{g}_{BV(P)} = \norm{f}_{BV(\tau)}$.

Fix $\epsilon > 0$. Let $f_1  = f  - g|\sigma$. Clearly $f \equiv 0$ on $\tau$,  $f_1|\sigma_1 \in AC(\sigma_1)$ and $f_1|\sigma_2 \in AC(\sigma_2)$.

 As $f_1| \sigma_1 \in AC(\sigma_1)$, by Lemma~3.2 and Theorem~6.1 of \cite{DLS2} there exists $h_1 \in \CTPP(P_1)$ with $\norm{f_1 - h_1}_{BV(\sigma_1)} \le \epsilon/8$. 

Now $h_1| \tau_0$ is piecewise linear, and in particular it is absolutely continuous. We can therefore extend $h_1$ to $P_2$ by setting $h_1(x,y) = h_1(x,0)$ for all $(x,y) \in P_2$. Note that $h_1 \in \CTPP(P)$. Now
  \begin{align*} \norm{h_1}_{BV(\sigma_2)} = \norm{h_1}_{BV(\tau_0)}
       &= \norm{f_1 - h_1}_{BV(\tau_0)}  \\
       &= \norm{f_1 - h_1}_{BV(\tau)}
       \le \norm{f_1 - h_1}_{BV(\sigma_1)} < \epsilon/8.
  \end{align*}
Similarly, construct $h_2 \in \CTPP(P_2)$ with $\norm{f_1 - h_2}_{BV(\sigma_2)} < \epsilon/8$ and extend this to all of $P$ as above, so that $\norm{h_2}_{BV(\sigma_1)} < \epsilon/8$. 

Let $h = h_1 + h_2 \in \CTPP(P)$. Then
  \[ \norm{f_1 - h}_{BV(\sigma_1)} \le \norm{f_1-h_1}_{BV(\sigma_1)} + \norm{h_2}_{BV(\sigma_1)}
       < \epsilon/4.\]
Similarly, $\norm{f_1 - h}_{BV(\sigma_2)} < \epsilon/4$ and so, by Theorem~\ref{join-line}, $\norm{f_1 - h}_{BV(\sigma)} < \epsilon$. Thus, $f_1 \in AC(\sigma)$. Since $g|\sigma \in AC(\sigma)$, this implies that $f \in AC(\sigma)$.
\end{proof}

 The next theorem says that we can join functions which are absolutely continuous on polygonal regions.

\begin{theorem}\label{patching-polygons}
Let $Q_1,\dots,Q_n$ be disjoint bounded open sets whose boundaries consist of a finite number of line segments. Let $P_i = \cl(Q_i)$, $1 \le i \le n$, and let $\sigma = \cup_{i=1}^n P_i$. Suppose that $f: \sigma \to \mC$. 
Then $f \in AC(\sigma)$ if and only if $f|P_i \in AC(P_i)$ for each $i$.
\end{theorem}

\begin{proof} The forward implication is immediate.

Suppose that $f \in AC(P_i)$ for each $i$. We shall use the Patching Lemma to show that $f \in AC(\sigma)$.

Suppose that $\vecx \in \sigma$. There are three mutually exclusive cases to consider.
\begin{enumerate}
  \item $\vecx$ lies in a single set $P_i$. In this case there is certainly compact neighbourhood of $\vecx$ in $\sigma$ on which $f$ is absolutely continuous.
  \item $\vecx$ lies in the boundary of two regions, say $P_1$ and $P_2$, but not at a corner of either of the region. In this case one may choose a small closed disk $D \subseteq P_1 \cup P_2$ centred at $\vecx$ such that $f$ is absolutely continuous on $D \cap P_1$ and on $D \cap P_2$. By Theorem~\ref{pasting-lemma-ac}, $f$ is then absolutely continuous on $D$, which is a compact neighbourhood of $\vecx$ in $\sigma$.
  \item $\vecx$ lies in $m \ge 2$ regions, say $P_1,\dots, P_m$, and it lies at a corner of at least one of these.
Since we are only interested in the local behaviour near $\vecx$, by intersecting with a small square around $\vecx$ we may assume that each region is in fact a polygon. 
     If $\vecx$ lies in $P_i$ but not at a corner, one can split $P_i$ into two smaller polygons each with a corner at $\vecx$, and of course $f$ is absolutely continuous on each of the smaller polygons. 
   We may therefore also assume that $\vecx$ lies at a corner of each of $P_1,\dots,P_m$, and by again splitting a polygon if necessary, we can assume that the polygons can be labelled so that 
   \begin{itemize}
   \item each $P_i$ is convex,
   \item $P_1,\dots,P_{r}$ are in one closed half-plane, and that $P_{r+1},\dots,P_m$ are in the complementary closed half-plane (see Figure~\ref{poly-point}). (Note that it may be that all the polygons lie in a single half-plane.)
   \end{itemize}
 One may now inductively apply Theorem~\ref{pasting-lemma-ac} to show that the restriction of $f$ to $\cup_{i=1}^r P_i$ and $\cup_{i=r+1}^m P_i$ are both absolutely continuous. Applying this theorem one more time shows that $\cup_{i=1}^m P_i$ is a compact neighbourhood of $\vecx$ on which $f$ is absolutely continuous.
 \end{enumerate}
In all cases, there is a compact neighbourhood of $\vecx$ on which $f$ is absolutely continuous, and so $f \in AC(\sigma)$ by the Patching Lemma. 
\end{proof}

  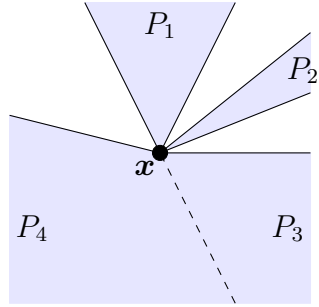
\begin{figure}[ht!]
\begin{center}
\begin{tikzpicture}[scale=1]

\draw[fill,blue!10] (-1,2) -- (0,0) -- (1,2) -- (-1,2);
\draw[black]  (-1,2) -- (0,0) -- (1,2);
\draw[fill,blue!10] (2,1.6) -- (0,0) -- (2,0.8) -- (2,1.6);
\draw[black] (2,1.6) -- (0,0) -- (2,0.8);
\draw[fill,blue!10] (-2,0.5) -- (0,0) -- (2,0) -- (2,-2) -- (-2,-2) -- (-2,0.5);
\draw[black] (-2,0.5) -- (0,0) -- (2,0);

\draw[black,dashed]  (1,-2) -- (0,0); 

\draw (0,1.7) node {$P_1$};
\draw (1.9,1.1) node {$P_2$};
\draw (1.7,-1) node {$P_3$};
\draw (-1.7,-1) node {$P_4$};
\draw[black,fill] (0,0) circle (0.1cm);
\draw (-0.2,-0.2) node {$\vecx$};
\end{tikzpicture}
\end{center}
\caption{Diagram for Theorem~\ref{patching-polygons}. Here the bottom polygon is split in two so that $P_1,P_2,P_3$ lie in one half-plane, and $P_4$ lies in the complementary half-plane.}
\label{poly-point}
\end{figure}

As noted earlier, in the above theorem, it is not possible to bound the norm of $f$ by some absolute constant times the sum of the norms of the restrictions of $f$ to the components $P_i$. On the other hand, one can give a bound which will depend on the geometry of the components.

\begin{theorem}\label{add-triangle}
Let $P \in \PWP$  and let $T$ be a closed triangular whose interior is disjoint from $P$. Let $\sigma = P \cup T$. Then there is a constant $K$ such that if $f: \sigma \to \mC$ then
  \[ \norm{f}_{BV(\sigma)} \le K\bigl( \norm{f}_{BV(P)} + \norm{f}_{BV(T)}\bigr). \]
\end{theorem}

\begin{proof}
Extending the sides of $T$ allows us to split the exterior of $T$ into 3 sectors $R_1,R_2,R_3$ whose vertices lie at vertices of $T$, as in Figure~\ref{add-triangle-fig}. Applying Theorem~\ref{join-sectors} shows that
  \[ \norm{f}_{BV(T \cap R_1)} \le 8 \bigl(\norm{f}_{BV(T)} + \norm{f}_{BV(R_1 \cap P)} \bigr)
  \le 8 \bigl(\norm{f}_{BV(T)} + \norm{f}_{BV(P)} \bigr). \]
Applying Theorem~\ref{join-sectors} twice more shows that
  \[ \norm{f}_{BV(\sigma)} \le 8^3\bigl(\norm{f}_{BV(T)} + \norm{f}_{BV(P)} \bigr) \]
as required. 
\end{proof}

\begin{figure}
\begin{center}
\begin{tikzpicture}[scale=1]

\draw[fill,green!10] (0,0) -- (4,0) -- (4,1) -- (1,1) -- (1,3.5) -- (3.6,3.5) -- (3.6,3) -- (2.5,3) -- (2.5,2.7) -- (4,2.7) -- (4,4) -- (0,4) -- (0,0);
\draw[green]   (0,0) -- (4,0) -- (4,1) -- (1,1) -- (1,3.5) -- (3.6,3.5) -- (3.6,3) -- (2.5,3) -- (2.5,2.7) -- (4,2.7) -- (4,4) -- (0,4) -- (0,0);

\draw[fill,blue!10] (2.5,3.5) -- (1,2) -- (2.5,1) -- (2.5,3.5);
\draw[blue] (2.5,3.5) -- (1,2) -- (2.5,1) -- (2.5,3.5); 

\draw[red,dashed] (2.5,-0.7) -- (2.5,3.5); 
\draw[red,dashed] (-0.7,3.13333) -- (2.5,1); 
\draw[red,dashed] (1,2) -- (3.8,4.8); 

\draw (2,2.1) node {$T$};
\draw(0.5,3.5) node {$P$};

\draw (2,4.5) node {$R_1$};

\draw (4,2) node {$R_2$};

\draw (-0.6,1) node {$R_3$};
\end{tikzpicture}
\caption{Joining a triangle to a polygonal region.}\label{add-triangle-fig}
\end{center}
\end{figure}
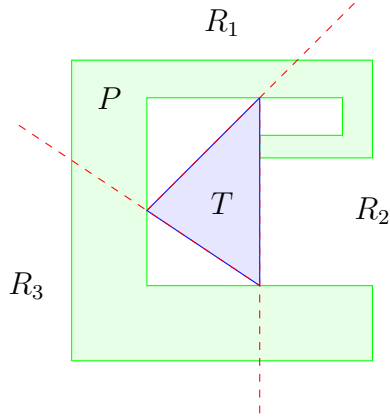

\begin{corollary}\label{join-two-polys}
Suppose that $P_1 \in \PWP$ and that $P_2$ is a closed polygonal region in the plane whose interior is disjoint from $P_1$, Let $\sigma = P_1 \cup P_2$. Then there exists a constant $K(P_1,P_2)$ such that if $f: \sigma \to \mC$ then
  \[ \norm{f}_{BV(\sigma)} \le K(P_1,P_2) \bigl(\norm{f}_{BV(P_1)}+ \norm{f}_{BV(P_2)} \bigr). \]
\end{corollary}
  
\begin{proof} 
Triangulate $P_2$ as $\cup_{k=1}^n T_k$. Applying the previous theorem repeatedly then shows that
  \begin{align*} \norm{f}_{BV(\sigma)}
   & \le K^n \bigl(\norm{f}_{BV(P_1)} + \norm{f}_{BV(T_1)} + \dots + \norm{f}_{BV(T_n)} \bigr) \\
   & \le  K^n \bigl(\norm{f}_{BV(P_1)} + n \norm{f}_{BV(P_2)}\bigr). \qedhere
   \end{align*}
\end{proof} 

 One can of course extend this result to deal with any finite number of polygons. As can be seen from the proof, the constant  $K(P_1,P_2)$ can be chosen  so that it only depends on the minimum number of triangles needed to triangulate one of the the polygons. (The constant needs to depend at least linearly in this number, but it seems unlikely that it needs the dependence given in the proof.)

The next result says that if we can extend a function to be absolutely continuous on the `holes' of a $\PWP$ set, then the function will be absolutely continuous on the filled in set.

\begin{theorem}\label{fill-in-pwp} 
Suppose that $\sigma_0 \in \PWP$. Let $W_1,\dots,W_d$ denote the bounded components of the complement of $\sigma_0$, and let $\sigma  = \sigma_0 \cup \bigl(\cup_{i=1}^d \cl(W_i)\bigr)$. Then $f \in AC(\sigma)$ if and only if $f| \sigma_0 \in AC(\sigma_0)$ and $f|\cl(W_i) \in AC(\cl(W_i))$ for $1 \le i \le d$. 
\end{theorem}

\begin{proof} One only needs to prove the converse direction. Suppose then that $f| \sigma_0 \in AC(\sigma_0)$ and $f|\cl(W_i) \in AC(\cl(W_i))$ for $1 \le i \le d$. 
One can triangulate $\sigma_0$ and each of the sets $W_i$. The set $\sigma$ is then a union of triangles and the restriction of $f$ to each of these is absolutely continuous so the result follows from Theorem~\ref{patching-polygons}.
\end{proof}

At this point we have enough to show that if $\sigma_0$ is a polygonal region with finitely many polygonal windows, then, using Theorem~\ref{fill-poly}, we can extend the function to those windows and hence produce an absolutely continuous function on the `filled-in' set. An alternative proof of this result may be found in \cite[Theorem~6.4.2]{St}, where an explicit bound on the norm of the extension is given in terms of the total number of edges on the polygonal region and its windows. 

It is worth noting that in all the known examples where a joined function fails to be absolutely continuous, the issue has been a failure to control the variation norm. This leads to the following open question.

\begin{question}
If $f_1 \in AC(\sigma_1)$ and $f_2 \in AC(\sigma_2)$ and $f \in BV(\sigma)$, must we have $f \in AC(\sigma)$?
\end{question}


\section{An $\APIC$ joining theorem}\label{S:apic-join}

Theorem~\ref{pasting-lemma-ac} is strong enough to prove that a function $f$ defined on a $\PIC$ set $\sigma = \cup_{k=1}^n c_k$ is absolutely continuous if and only if it is absolutely continuous on each of its component curves. Our next task is to extend this result to the class of $\APIC$ sets. 
Note that as is shown in Theorem~6.2.1 of \cite{St}, if $c_1 = \{(x,x^2) \st 0 \le x \le 1\}$ and $c_2 = \{(x,2x^2) \st 0 \le x \le 1\}$, and $\sigma = c_1 \cup c_2$ then one can find a function which is absolutely continuous on each of the components, but which is not absolutely continuous on $\sigma$, so one cannot extend this result to arbitrary unions of convex curves. (Of course, this doesn't imply that one can't extend absolutely continuous functions from such sets.)

We begin by giving showing that one can join functions of bounded variation on different components of an $\APIC$ set.

\begin{lemma}\label{APIC-bound-lemma}
Let $P$ be a convex $N$-gon with sides $\ell_1,\dots,\ell_N$ and let $c$ be a convex curve in $P$. Suppose that $\calC$ is a nonempty subset of $\{c,\ell_1,\dots,\ell_N\}$ and let $\sigma = \cup\{\tau \st \tau \in \calC\}$. 
If $f: \sigma \to \mC$ then 
  \[
  \max \bigl\{\norm{f}_{BV(\tau)} \st \tau \in \calC \bigr\}
      \le \norm{f}_{BV(\sigma)} 
      \le K_N \Bigl( 
         \sum_{\tau \in \calC} \norm{f}_{BV(\tau)} \Bigr).
      \]
\end{lemma}

\begin{proof}
The left hand inequality is clear, so it remains to establish the equality on the right. 
Label the sets in $\calC$ as $\tau_0,\dots,\tau_r$, where, if $c \in \calC$ we choose $\tau_0 = c$.

Let $\sigma_0 = \tau_0$ and for $j = 1,\dots, r$ let $\sigma_j = \sigma_{j-1} \cup \tau_j$. Since $P$ is convex, one may apply Theorem~\ref{join-line} to show that for each $j$
  \[ \norm{f}_{BV(\sigma_j)} \le 2 \bigl(\norm{f}_{BV(\sigma_{j-1})} + \norm{f}_{BV(\tau_j)} \bigr) \]
from which the result follows easily.
\end{proof}

\begin{theorem}\label{APIC-join-bound}
Suppose that $\sigma \in \APIC$. Let $\calC$ denote the components of $\sigma$ with respect to a polygonal mosaic $\calM$. 
Then there exists $K_{\sigma,\calM}$ such that if $f: \sigma \to \mC$ then
  \[ \norm{f}_{BV(\sigma)} \le 
    K_{\sigma,\calM} \sum_{\tau \in \calC} \norm{f}_{BV(\tau)} . \]
\end{theorem}

\begin{proof}
Let $\calM = \{P_i\}_{i=1}^n$ be the polygonal mosaic. 
As above, for $j = 1,\dots,n$, let $\sigma_j = \cup_{i=1}^j (P_i \cap \sigma) $ and let $\calC_j$ denote the components of $\sigma$ which lie in $\sigma_j$. The previous lemma implies that
$\norm{f}_{BV(\sigma_1)}$ is bounded by a constant $K_1$ (depending on the number of sides of $P_1$) times the sum of the norms of $f$ on the components of $\sigma$ in $\calC_1$. 

Suppose now that $1 \le j < n$ and that
  \begin{equation}\label{apic-induction}
     \norm{f}_{BV(\sigma_j)} \le K_j \sum_{s \in \calC_j} \norm{f}_{BV(s)}. 
  \end{equation}
Let $r$ denote the number of sides of $P_{j+1}$. Extending the sides of $P_{j+1}$ one can split the exterior of $P_{j+1}$ into $r$ regions $R_1,\dots,R_r$, each contained in a sector whose vertex is one of the vertices of the polygon, as in Figure~\ref{apic-bound-pic}.

As in the proof of Theorem~\ref{add-triangle}, one can now repeatedly apply Theorem~\ref{join-sectors} to show that 
  \[ \norm{f}_{BV(\sigma_{k+1})} \le 8^r \big( \norm{f}_{BV(\sigma_k)} + \norm{f}_{BV(\sigma \cap P_{k+1})} \bigr). \]
Applying Lemma~\ref{APIC-bound-lemma} and using (\ref{apic-induction}) gives the result.
\end{proof}

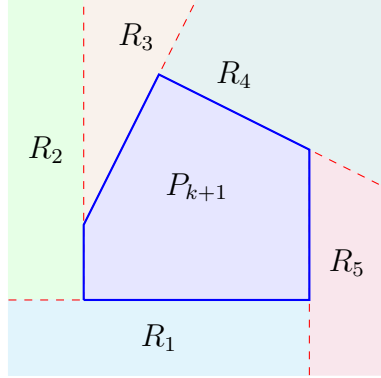
\begin{figure}
\begin{center}
\begin{tikzpicture}[scale=1]

\draw[fill,green!10] (0,1) -- (1,1) -- (1,5) -- (0,5) -- (0,1);
\draw[fill,brown!10] (1,2) -- (2.5,5) -- (1,5) -- (1,2);
\draw[fill,teal!10] (2,4) -- (5,2.5) -- (5,5) -- (2.5,5) -- (2,4);
\draw[fill,purple!10] (4,3) -- (4,0) -- (5,0) -- (5,2.5) -- (4,3);
\draw[fill,cyan!10] (4,1) -- (0,1) -- (0,0) -- (4,0) -- (4,1);

\draw[fill,blue!10] (1,1) -- (4,1) -- (4,3) -- (2,4) -- (1,2);
\draw[blue]   (1,1) -- (4,1) -- (4,3) -- (2,4) -- (1,2);

\draw[red,dashed] (0,1) -- (4,1);
\draw[red,dashed] (4,0) -- (4,3);
\draw[red,dashed] (2,4) -- (5,2.5);
\draw[red,dashed] (1,2) -- (2.5,5);
\draw[red,dashed] (1,1) -- (1,5);
\draw[blue,thick]   (1,1) -- (4,1) -- (4,3) -- (2,4) -- (1,2) -- (1,1);

\draw (2.5,2.5) node {$P_{k+1}$};
\draw (2,0.5) node {$R_1$};
\draw (0.5,3) node {$R_2$};
\draw (1.7,4.5) node {$R_3$};
\draw (3,4) node {$R_4$};
\draw (4.5,1.5) node {$R_5$};
\end{tikzpicture}
\caption{Splitting the exterior of a convex polygon into sectors $R_1,\dots,R_r$.}\label{apic-bound-pic}
\end{center}
\end{figure}

Our next step is to ensure that a function $f$ defined on an $\APIC$ set is absolutely continuous if and only if it is absolutely continuous on each component curve. Recall from Example~\ref{bad-join} that this result fails if one considers more general finite unions of convex curves.

\begin{lemma}\label{LT-vertex-join}
Suppose that $T$ is a triangle in $\mR^2$ with vertices at $\vecx$, $\vecy$ and $\vecz$. Let $c$ be a differentiable convex curve joining $\vecx$ to $\vecy$ through the interior of $T$ which meets the side $\ell = \ls[\vecx,\vecz]$ tangentially at $\vecx$.
Let $\sigma = c \cup \ell$ and suppose that $f: \sigma \to \mC$. If $f|c \in AC(c)$ and $f|\ell \in AC(\ell)$ then $f \in AC(\sigma)$.
Furthermore $f$ has an extension $\hat f \in AC(T)$ with 
  \[ \ssnorm{\hat{f}}_{BV(T)} \leq 3 \norm{f}_{BV(\sigma)}. \]
\end{lemma}

\begin{proof}
By applying a suitable affine transformation we may assume that $\vecx = (0,0)$, that $\ell = \{(0,y) \st 0 \le y \le 1\}$, and that $c$ joins $(0,0)$ to $(1,1)$. We may assume that $c$ is the graph of the convex function $k:[0,1] \to \mR$. Let $R = [0,1] \times [0,1]$. 

Let $g(x,y) = f(0,y)$, $(x,y) \in R$. Then, as $f|\ell \in AC(\ell)$, Theorem~\ref{ac-properties}\ref{r-ext} implies that $g \in AC(R)$. Define $f_1: c \to \mC$ by $f_1(x,y) = f(x,y) - g(x,y)$, $(x,y) \in c$. Now $f_1 \in AC(c)$, so as $c$ is a convex curve, the function $h:[0,1] \to \mC$, $h(x) = f_1(x,k(x))$ is also absolutely continuous. Indeed ${\hat f}_1(x,y) = h(x)$ is then absolutely continuous on all of $R$. It follows that
  \[ {\hat f}(x,y) = {\hat f}_1(x,y)  + g(x,y), \qquad (x,y) \in R \]
  is also absolutely continuous on $R$. But of course ${\hat f}$ is an extension of $f$ and so $f = {\hat f}|\sigma \in AC(\sigma)$.
  
  The estimate for $\|\hat{f}\|_{BV(T)}$ simply follows from the triangle inequality and the facts that $\|\hat{f}_1\|_{BV(T)} \le 2\|f\|_{\sigma}$ (by Lemma~\ref{from-c-to-T}) and $\|g\|_{BV(T)} \le \|f\|_{BV(\sigma)}$ (by Theorem~\ref{ac-properties}\ref{r-ext}).
\end{proof}

The following theorem is a generalization of \cite[Theorem~6]{AS2} for $\PIC$ sets.

\begin{theorem}\label{APIC-joining}
Suppose that $\sigma \in \APIC$ is represented as above as a union of convex curves and lines meeting these curves tangentially, and that $f: \sigma \to \mC$. Then $f \in AC(\sigma)$ if and only if
 \begin{enumerate}
  \item $f|c_k \in AC(c_k)$ for $k = 1,\dots,n$, and
  \item $f|\ell_j \in AC(\ell_j)$ for $j = 1,\dots,m$.
 \end{enumerate}
\end{theorem}
 
\begin{proof}
The forward direction follows from the general results about restrictions of absolutely continuous functions.

Suppose then that (1) and (2) hold. By the Patching Lemma, it suffices to show that $f$ is absolutely continuous on a compact neighbourhood of each point in $\sigma$.
Note that if $\vecx \in \sigma$ is not a line tangential vertex, then, by the $\PIC$ joining theorem \cite[Corollary 3]{AS2}, $f$ is absolutely continuous on a compact neighbourhood of $\vecx$. So suppose that $\vecx$ is a line tangential vertex.

Fix a small polygonal region $P$ centred at $\vecx$ and split this into triangles $T_1,\dots,T_r$ so that (the closure of) each triangle contains at most one curve $c_k$ (and necessarily any line which meets $c_k$ tangentially at $\vecx$). (See Figure~\ref{triangulate-LT}.) Using Lemma~\ref{LT-vertex-join}, one can deduce that $f$ is absolutely continuous on $T_j \cap \sigma$ for all $j$. 

Again, as in the proof of Theorem~\ref{patching-polygons}, one may inductively use Theorem~\ref{pasting-lemma-ac} to show that $f$ is absolutely continuous on the compact neighbourhood $P \cap \sigma$ of $\vecx$. Thus $f$ is absolutely continuous on all of $\sigma$.
\end{proof}

\begin{figure}
\begin{center}
\begin{tikzpicture}[scale=4]

\draw[fill,green!10] (10,1) -- (10.25,0.5) -- (10.5,1.5) -- (10,1);

\draw[fill,blue!10] (10,1) -- (10.5,1.5) -- (10,1.5) -- (10,1);
\draw[blue] (10,1) -- (10.5,1.5) -- (10,1.5) -- (10,1);

\draw[fill,green!10] (10,1) -- (10,1.5) -- (9.5,1.5) -- (10,1);

\draw[fill,blue!10] (10,1) -- (9.5,1.5) -- (9.5,0.5) -- (10,1);
\draw[blue]  (10,1) -- (9.5,1.5) -- (9.5,0.5) -- (10,1);


\draw[fill,blue!10] (10,1) -- (9.5,0.5) -- (10,0.5) -- (10,1);
\draw[blue]  (10,1) -- (9.5,0.5) -- (10,0.5) -- (10,1);


\draw[fill,green!10] (10,1) -- (9.5,0.5) -- (10,0.5) -- (10,1);
\draw[blue]  (10,1) -- (9.5,0.5) -- (10,0.5) -- (10,1);

\draw[fill,blue!10] (10,1) -- (10,0.5) -- (10.25,0.5) -- (10,1);
\draw[blue] (10,1) -- (10,0.5) -- (10.25,0.5) -- (10,1);

\draw[blue] (10.5,1.5) -- (9.5,1.5) -- (9.5,0.5) -- (10.25,0.5) -- (10.5,1.5);

\draw[ultra thick] (10,0.5) -- (10,1.5);
\draw[ultra thick] (10,1) -- (10.5,0.5);
\draw[ultra thick] (10,1) arc (0:30:1cm); 
\draw[ultra thick] (10,1) arc (180:150:1cm); 
\draw[ultra thick] (10,1) arc (225:260:1cm);
\draw[ultra thick] (10,1) arc (180:210:1cm);
\draw[ultra thick] (10,1) -- (9.5,1);
\draw[fill,black] (10,1) circle (0.03cm);
\draw (10.08,1) node {$\vecx$};
\end{tikzpicture}
\caption{Triangulating around a  line tangential vertex.}\label{triangulate-LT}
\end{center}
\end{figure}
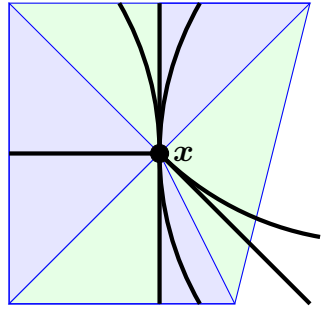

A consequence of the preceding results is that $\APIC$ is a `Gelfand--Kolmogorov family' of compact sets, which generalizes \cite[Theorem~7]{AS2}. 

\begin{corollary}
Suppose that $\sigma,\tau \in \APIC$. Then $AC(\sigma)$ is isomorphic (as a Banach algebra) to $AC(\tau)$ if and only if $\sigma$ is homeomorphic to $\tau$. 
\end{corollary}

A final result which we will need is to be able to extend an absolutely continuous function defined on an $\APIC$ set on each polygon of its polygonal mosaic.

\begin{theorem}\label{c-in-poly-thm}
Suppose that $c$ is a projectable convex curve joining two vertices $\vecx,\vecy$ of a convex polygon $P$ to each other through the interior of $P$. Let $\sigma_0$ be a compact set which is the union of $c$ and a (possibly empty) subset of the boundary of $P$, and suppose that $f \in AC(\sigma_0)$. Then $f$ admits an extension $\hat f \in AC(P)$.
\end{theorem}

\begin{proof}
Note that $f$ is initially specified at at least two points of the boundary of $P$, namely the endpoints of $c$. 
First extend $f$ to all of the sides of $P$ by making it continuous and affine on any parts of $\partial P$ which were not part of $\sigma_0$ (and hence, by Lemma~\ref{iota-lemma} and Theorem~\ref{APIC-joining}, it is in $AC(\partial P)$). If $c$ is a straight line, then the conclusion follows by applying Theorems \ref{fill-poly} and \ref{patching-polygons}. 

Suppose then that $c$ is not a straight line. Let $P_1 \subseteq P$ be the convex polygon containing $c$ which has the diagonal $\ls[\vecx,\vecy]$ as a side. (If $\ls[\vecx,\vecy]$ was already a side, then $P_1 = P$.)
Since $c$ is convex, one can find a line segment $\ell_1$ in $P_1$ which is tangential to $c$ at $\vecx$. Choose a point $\vecz$ on $c$ so that the orthogonal projection $\vecw$ of $\vecz$ lies on $\ell_1$, and consider the triangle $T_1$ with vertices $\vecx,\vecz,\vecw$. If $\ell_1$ forms part of a side of $P$, and $f$ has been specified on that side, then extend $f$ to all of $T_1$ using Lemma~\ref{LT-vertex-join}. Otherwise extend $f$ to $T_1$ using Lemma~\ref{from-c-to-T}. In either case one has an absolutely continuous extension of $f$ to this triangle. 

The same procedure can be repeated choosing a small right-angle triangle $T_2$ with vertices $\vecy,\vecz',\vecw'$ to define an extension of $f$ on $T_2$. The points $\vecz,\vecz' \in c$ now lie in the interior of $P$. Let $R$ be a rectangle with side $\ls[\vecz,\vecz']$ which contains $c'$, the part of $c$ between $\vecz$ and $\vecz'$. Choose a polygon $Q \subseteq R \cap P_1$ which contains $c'$.
Since $c$ is projectable we may use Theorem~\ref{ac-properties}\ref{r-ext} to extend $f$ to all of $Q$ by making it constant on lines which are orthogonal to $\ls[\vecz,\vecz']$. 

We have now specified the extension $\hat f$ on the green regions in Figure~\ref{c-inside-poly}, as well as on the blue boundary of $P$. This leaves a finite number of polygonal regions on which $\hat f$ given on the boundary, but not in the interiors of each of the regions. Using Theorem~\ref{fill-poly} one can now extend $\hat f$ so that it is absolutely continuous on each of these regions. Finally, we can apply Theorem~\ref{patching-polygons} to deduce that $\hat f$ is absolutely continuous on the whole set $P$.
\end{proof}

\begin{figure}
\begin{center}
\begin{tikzpicture}[scale=1.7]

\coordinate (x) at (-1.2,0);
\coordinate (y) at (4,0);
\coordinate (z) at (-0.28,1);
\coordinate (w) at (-0.52,1.16);
\coordinate (zd) at (3.05,1);
\coordinate (wd) at (3.29,1.17);
\coordinate (r1) at (-0.28,2);
\coordinate (r2) at (3.05,2);
\coordinate (q1) at (0,1.8);
\coordinate (q2) at (2.4,1.8);

\draw[green!10,fill] (x) -- (z) -- (w) -- (x);
\draw[green!10,fill] (y) -- (wd) -- (zd) -- (y);
\draw[green!10,fill] (z) -- (zd) -- (q2) -- (q1) -- (z);

\draw[red,line width = 1mm] (4,0) arc (30:150:3cm);
\draw[blue, thick] (x) -- (1,-1) -- (3,-1) -- (y) -- (2.8,2) -- (1,3) -- (0,2.3) -- (-2,1) -- (x);
\draw[red,line width = 1mm] (3,-1) -- (y) -- (2.8,2);
\draw[green,dashed] (x) -- (y);
\draw[green] (x) -- (z) -- (w) -- (x);
\draw[green] (y) -- (zd) -- (wd);
\draw[blue,dashed] (zd) -- (r2) -- (r1) -- (z);
\draw[green] (z) -- (zd) -- (q2) -- (q1) -- (z);

\draw[fill,black] (x) circle (0.03cm);
\draw (x) node[left] {$\vecx$};
\draw[fill,black] (y) circle (0.03cm);
\draw (y) node[right] {$\vecy$};
\draw[fill,black] (z) circle (0.03cm); 
\draw (z) node[below] {$\vecz$};
\draw[fill,black] (zd) circle (0.03cm); 
\draw (zd) node[below] {$\vecz'$};
\draw[fill,black] (w) circle (0.03cm); 
\draw (w) node[left] {$\vecw$};
\draw[fill,black] (wd) circle (0.03cm); 
\draw (wd) node[right] {$\vecw'$};
\draw (-1,0.6) node {$T_1$};
\draw (3.85,0.6) node {$T_2$};
\draw (0.1,1.6) node {$Q$};
\draw (1.1,2.7) node {$P_1$};
\draw (1.4,1.35) node {$c$};
\draw (r2) node[right] {$R$};

\end{tikzpicture}
\caption{The construction in the proof of Theorem~\ref{c-in-poly-thm}. The value of $f$ is initially given on the red curves.}\label{c-inside-poly}
\end{center}
\end{figure}
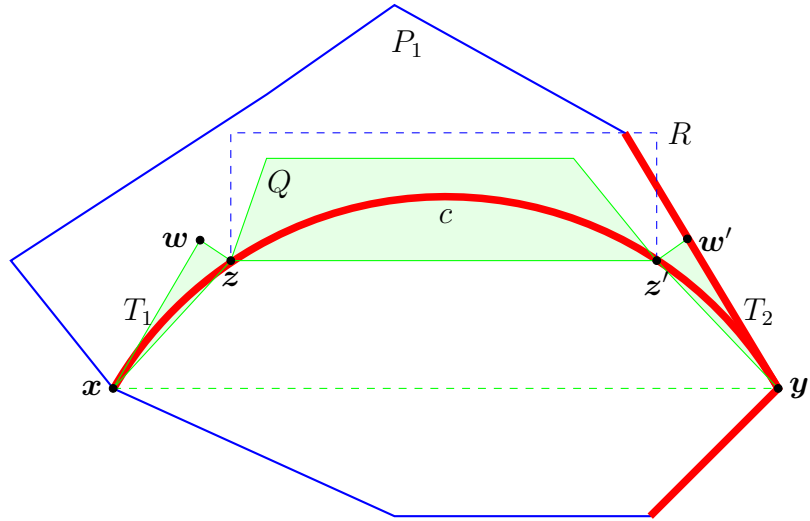

\section{The $\SPIC$ extension theorem}

Suppose that $\sigma_0 \in \SPIC$. Then there exists a set $\sigma_1 \in \APIC$ with $\sigma_0 \subseteq \sigma_1$. By Lemma~\ref{split-projectable}, we may choose a polygonal mosaic so that each of the components of $\sigma_1$ is projectable, and we shall assume that that is the case throughout this section.

We will see below that we can always construct extensions of $AC$ functions defined on APIC sets, and then extend this result to the case of SPIC sets. 

\begin{lemma}\label{extend-to-apic}
Suppose that $\emptyset \ne \sigma_0 \subseteq \sigma_1$ with $\sigma_1 \in \APIC$. If $f \in AC(\sigma_0)$, then there exists an extension $\hat f \in AC(\sigma_1)$ with $\ssnorm{\hat f}_{BV(\sigma_1)} \le C_{\sigma_1} \norm{f}_{BV(\sigma_0)}$ for some $C_{\sigma_1} > 0$.
\end{lemma}

\begin{proof} Let $\calC$ denote the set of components of $\sigma_1$ with respect to a suitable polygonal mosaic.
Suppose that $f \in AC(\sigma_0)$.
Let $\sigma_0'$ denote the union of $\sigma_0$ and all of the (finitely many) endpoints of the components of $\sigma_1$. Any endpoint which is not in $\sigma_0$ is an isolated point of $\sigma_0'$. Thus we can extend $f$ to $\sigma_0'$ by setting it to be zero at any such point and this extension is in $AC(\sigma_0')$. The norm of the extension may increase by a factor of 3 on each component.

Let $\tau \in \calC$ be a component of $\sigma_1$.
By Theorem~\ref{conv-curve-extend} we can extend $f$ from $\sigma_0' \cap \tau$ to ${\hat f} \in AC(\tau)$ (with a possible doubling of the norm). 
Since $\sigma_0'$ contains all the endpoints of the components of $\sigma_1$, doing this for each $\tau \in \calC$ produces a well-defined function $\hat{f}: \sigma_1 \to \mC$. By Theorem~\ref{APIC-joining}, $\hat{f} \in AC(\sigma_1)$. 

The existence of the constant $C_{\sigma_1}$ follows from Theorem~\ref{APIC-join-bound}.
\end{proof}

\begin{theorem}\label{spic-extend}
Suppose that $\sigma_0 \in \SPIC$ and that $\sigma$ is a compact superset of $\sigma_0$. If $f \in AC(\sigma_0)$ then there exists an extension $\hat f \in AC(\sigma)$ with $\ssnorm{\hat f}_{BV(\sigma)} \le K_{\sigma_0} \norm{f}_{BV(\sigma_0)}$ for some $K_{\sigma_0} > 0$.
\end{theorem}

\begin{proof} 
Suppose that $f \in AC(\sigma_0)$.
As $\sigma_0 \in \SPIC$, there exists $\sigma_1 \in \APIC$ with $\sigma_0 \subseteq \sigma_1$. By Lemma~\ref{extend-to-apic} we can find an absolutely continuous extension of $f$ to $\sigma_1$. 

Let $\calM = \{P_i\}_{i=1}^n$ be a polygonal mosaic for $\sigma_1$ chosen so that each component convex curve is projectable. 
By Theorem~\ref{c-in-poly-thm} we can extend $f$ from $\sigma_1 \cap P_1$ to be absolutely continuous on $P_1$. We may now inductively further extend $f$ to each polygon $P_i$ in turn, taking into account the values which have already been determined on the boundary of the polygon at each stage. By Theorem~\ref{patching-polygons}, the extension $\hat f$ is absolutely continuous on $M = \cup_{i=1}^n P_i$ (the blue region in Figure~\ref{extend-M-to_R}).

Let $R_0$ be a large rectangle which contains $M$ in its interior. Let $\sigma_2 = M \cup \partial R_0$ and extend $\hatf$ to $\sigma_2$ by making it zero on the boundary of $\sigma_2$. As in \cite[Lemma~4.2]{DLS1}, $\hatf \in AC(\sigma_2)$ with $\ssnorm{\hatf}_{BV(\sigma_2)} \le 5 \norm{f}_{BV(M)}$.

The region $R_0 \setminus M$ can be triangulated by 
$\{T_i\}_{i=1}^m$
with each triangle having vertices in $\sigma_2$. If necessary, use Theorem~\ref{extend-to-apic} to extend $\hatf$ to the whole boundary of $T_1$. One can then use Theorem~\ref{fill-poly} to extend $\hatf$ to an absolutely continuous function on $T_1$. As above, one can repeat this procedure with each triangle in turn, taking into account the values of $\hatf$ already determined. 

If necessary, one now chooses an even larger rectangle $R$ containing $R_0 \cup \sigma$ and extends $\hatf$ further by setting $\hatf = 0$ on $R\setminus R_0$. 
Applying Theorem~\ref{patching-polygons} again shows that $\hatf \in AC(R)$. At each step we have a control over the norm of the extension. This provides a bound on $\ssnorm{\hatf}_{BV(R)}$ which depends only on the choice of $R_0$, and not on $\sigma$.

Finally we can restrict $\hat f$ to $\sigma$ to give the conclusion of the theorem.
\end{proof}

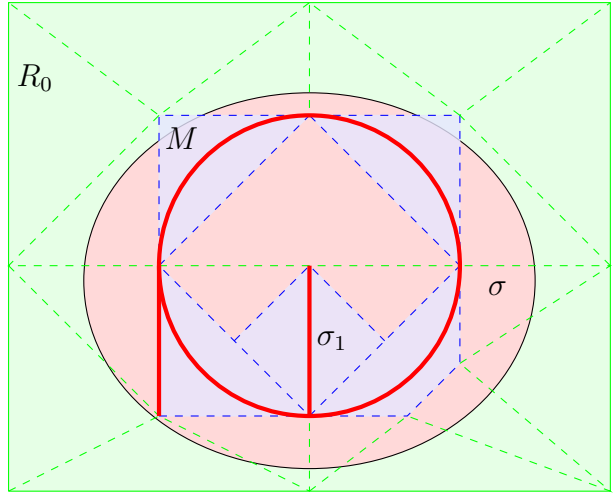
\begin{figure}
\begin{center}
\begin{tikzpicture}[scale=1]

\draw[fill,green!10] (-2,-1) -- (6,-1) -- (6,5.5) -- (-2,5.5) -- (-2,-1);
\draw[green] (-2,-1) -- (6,-1) -- (6,5.5) -- (-2,5.5) -- (-2,-1); 

\draw[fill,red!15] (2,1.8) ellipse (3cm and 2.5cm);
\draw[black] (2,1.8) ellipse (3cm and 2.5cm);

\draw[fill, blue!10,opacity=0.8] (0,0) -- (3.3,0) -- (4,0.7) -- (4,2) -- (3,1) -- (2,2) -- (1,1) -- (0,2) -- (0,0);
\draw[fill, blue!10,opacity=0.8]  (0,2) -- (2,4) -- (0,4) -- (0,2);
\draw[fill, blue!10,opacity=0.8]  (4,2) -- (4,4) -- (2,4) -- (4,2);

\draw[blue, dashed] (0,0) -- (2,0) -- (0,2) -- (0,0);
\draw[blue, dashed] (1,1) -- (2,2) -- (3,1);
\draw[blue, dashed] (2,0) -- (4,2) -- (4,0.7) -- (3.3,0) -- (2,0);
\draw[blue, dashed] (4,2) -- (4,4) -- (2,4) -- (4,2);
\draw[blue, dashed] (0,2) -- (2,4) -- (0,4) -- (0,2);

\draw[green,dashed] (-2,-1) -- (0,0) -- (2,-1) -- (2,0);
\draw[green,dashed] (2,-1) -- (3.3,0) -- (6,-1) -- (4,0.7);
\draw[green,dashed] (4,0.7) -- (6,2) -- (4,2);
\draw[green,dashed] (6,2)-- (4,4) -- (6,5.5);
\draw[green,dashed] (4,4) -- (2,5.5) -- (2,4);
\draw[green,dashed] (2,5.5) -- (0,4) -- (-2,5.5);
\draw[green,dashed] (0,4) -- (-2,2) -- (0,2);
\draw[green,dashed] (-2,2) -- (0,0);
\draw[green,dashed] (0,2) -- (4,2);

\draw[red,ultra thick] (2,2) circle (2cm);
\draw[red,ultra thick] (0,0) -- (0,2);
\draw[red,ultra thick] (2,0) -- (2,2);

\draw (-1.65,4.5) node {$R_0$};
\draw (4.5,1.7) node {$\sigma$};
\draw (2.3,1) node {$\sigma_1$};
\draw (0.3,3.7) node {$M$};

\end{tikzpicture}
\caption{Theorem~\ref{spic-extend}. The function $f$ is first extended to $\sigma_1 \in \APIC$, then to the blue polygons whose union is $M$, and then, one triangle at a time, to the rectangle $R_0$. The final step is to restrict to the set $\sigma$.}\label{extend-M-to_R}
\end{center}
\end{figure}

\section{The main theorem}

\begin{theorem}\label{main-theorem} Suppose that $\emptyset \ne \sigma_0 = \sigma_P \cup \sigma_S$ with $\sigma_P \in \PWP$ and $\sigma_S \in \SPIC$ and that $\sigma$ is a compact superset of $\sigma_0$. Then there exists a constant $C_{\sigma_0}$ such that for all $f \in AC(\sigma_0)$ there exists an extension $\hat{f} \in AC(\sigma)$ with $\|\hat{f}\|_{BV(\sigma)} \leq C_{\sigma_0}\|f\|_{BV(\sigma_0)}$.
\end{theorem}

\begin{proof}
Choose
a large rectangle $R$ which contains $\sigma$ in its interior. Let ${\hat \sigma}_P = \sigma_P \cup \partial R \in \PWP$. As in the proof of Theorem~\ref{spic-extend}, at the cost of a factor 5 in the norm, we can extend $f$ by setting it to be zero on  $\partial R$. 
Let $\mathcal{W} = \{W_1,\dots,W_d\}$ denote the connected components of $R \setminus {\hat \sigma}_P$. 

Fix $W \in \mathcal{W}$. Since the boundary of $W$ consists of a finite number of line segments, we can use Lemma~\ref{add-line-spic} to show that $\sigma_W = \partial W \cup (\sigma_S \cap W) \in \SPIC$.  Thus, by Theorem~\ref{spic-extend}, there is a constant $K_W$ and an extension $f_W$ of $f|\sigma_W$ to an absolutely continuous function on all of $\cl(W)$ with $\norm{f_W}_{BV(W)} \le K_W \norm{f}_{BV(\sigma_0)}$. 

Having done this for each $W \in \mathcal{W}$, setting
  \[ \hat f(\vecx) = \begin{cases} 
                           f(\vecx),   & \vecx \in \sigma_0, \\
                           f_W(\vecx),  & \vecx \notin \sigma_0,\ \vecx \in W \in \mathcal{W} 
                           \end{cases} \]
gives a function which is absolutely continuous on all of $R$ by Theorem~\ref{fill-in-pwp}. Finally we can restrict to $\sigma$ to obtain the required extension in $AC(\sigma)$. The bound follows from applying Corollary~\ref{join-two-polys} repeatedly.             
\end{proof}

Theorem~\ref{main-theorem} covers almost all the sets $\sigma_0$ for which it it is known that one can always find an absolutely continuous extension to a larger compact set. It is likely that the differentiability restriction on the convex curves considered can be relaxed. It is not too difficult, for example, to use the results of this paper to prove an extension result for a single convex curve consisting of an infinite collection of line segments.


It would of course be of interest to know whether one can always extend an absolutely continuous function on a disk to larger sets.

\section{An application to operator theory}

The motivation for this work is to be able to show that if a Banach space operator $T$ is an $AC(\sigma)$ operator, then it is in fact an $AC(\sigma(T))$ operator. That is, it has an $AC(\sigma(T))$ functional calculus. 

Suppose then that $T$ is bounded operator on a Banach space $X$ with $\sigma(T) = \sigma_P \cup \sigma_S$ with $\sigma_P \in \PWP$ and $\sigma_S \in \SPIC$. Suppose further than $T$ has an $AC(\sigma)$ functional calculus, which implies that $\sigma(T) \subseteq \sigma$. That is, there is a Banach algebra homomorphism $\Psi: AC(\sigma) \to B(X)$ which extends the natural polynomial functional calculus.

By Theorem~\ref{main-theorem} there is a map $e: AC(\sigma(T)) \to AC(\sigma)$ such that $e(f)$ is an extension of $f$ to $\sigma$. Thus the map $\Phi: AC(\sigma(T)) \to B(X)$ given by $\Phi(f) = \Phi(e(f))$ is well-defined. It is of course not immediately clear that this map is a functional calculus map for $T$ since the map $e$ has not been shown to have any algebraic properties.

Suppose then that $f,g \in AC(\sigma(T))$. Then $e(f+g) - (e(f)+e(g))$ is identically zero on $\sigma(T)$. Theorem 3.1.6 of \cite{CF} ensures that the support of the functional calculus map $\Psi$ is $\sigma(T)$, which implies that 
  \[ 0 = \Psi\big(e(f+g)-(e(f)+e(g))\big) = \Psi(e(f+g)) - \Psi(e(f)) - \Psi(e(g)) \]
or $\Phi(f+g) = \Phi(f) + \Phi(g)$. A similar proof shows that $\Phi$ is multiplicative. Since $\norm{\Phi(f)} \le \norm{\Psi} \norm{e(f)}_{BV(\sigma)}  \le C_{\sigma(T)} \norm{\Psi} \norm{f}_{\BV(\sigma(T))}$ we now have that $\Phi$ is a Banach algebra homomorphism. Finally, let $\lambda(z) = z$ be the identity function on the complex plane. Since $e(\lambda) - \lambda$ is identically zero on $\sigma(T)$ we must have that $\Phi(\lambda) = \Psi(e(\lambda)) = \Psi(\lambda) = T$, and so $\Phi$ is an $AC(\sigma(T))$ functional calculus for $T$.

We note in particular that this conclusion holds whenever $\sigma(T) \subset \mR$ or $\sigma(T) \subseteq \mT$, which covers the cases of well-bounded and trigonometrically well-bounded operators.

The remaining challenge is to be able to remove the above hypotheses on the spectrum of $T$. 

\subsection*{Acknowledgements}
The work of the second author was supported by the Research Training Program of the Department of Education and Training of the Australian Government. 

\subsection*{Rights Retention Statement} This research was produced in whole or part by UNSW Sydney researchers and is subject to the UNSW Intellectual property policy. For the purposes of Open Access, the authors have applied a Creative Commons Attribution CC-BY licence to any Author Accepted Manuscript(AAM) version arising from this submission.

%
%
\bibliographystyle{amsalpha}

\end{document}